\documentclass[letterpaper,leqno,10pt,twoside]{amsart}
\usepackage{amsmath,amsthm,amsfonts,amssymb,euscript,mathrsfs,graphics,color,latexsym,marginnote}
\usepackage[dvips]{graphicx}
\usepackage[margin=1in]{geometry}
\usepackage{hyperref}
\usepackage[toc,page]{appendix}
\usepackage{relsize}
\usepackage[shortlabels]{enumitem}
\usepackage[toc,page]{appendix}

\usepackage[dvipsnames]{xcolor}

\newcommand{\blue}[1]{\color{blue}}
\newcommand{\brown}[1]{\color{brown}}

\usepackage{hyperref}
\hypersetup{colorlinks=true, pdfstartview=FitV,linkcolor=blue!70!black,citecolor=red!70!black, urlcolor=green!60!black}
\definecolor{labelkey}{rgb}{0.6,0,0} 

\newtheorem{theorem}{Theorem}[section]
\newtheorem{lemma}[theorem]{Lemma}

\newtheorem{remark}[theorem]{Remark}

\makeatletter
\renewcommand \theequation {%
\ifnum \c@section>\z@ \@arabic\c@section.%
\fi\@arabic\c@equation} \@addtoreset{equation}{section}
\makeatother

\DeclareMathOperator*{\esssup}{ess\,sup}

\providecommand{\abs}[1]{\left\vert#1\right\vert}
\providecommand{\nm}[1]{\left\Vert#1\right\Vert}
\providecommand{\nnm}[1]{\left\vert\kern-0.25ex\left\vert\kern-0.25ex\left\vert#1\right\vert\kern-0.25ex\right\vert\kern-0.25ex\right\vert}

\providecommand{\br}[1]{\left\langle #1 \right\rangle}

\providecommand{\brv}[1]{\left\langle #1 \right\rangle_v}

\providecommand{\brr}[1]{\left( #1 \right)_{\mathcal{L}}}

\providecommand{\xnm}[1]{\left\Vert#1\right\Vert_{X}}

\providecommand{\pnm}[2]{\left\Vert#1\right\Vert_{L^{#2}}}
\providecommand{\tnm}[1]{\left\Vert#1\right\Vert_{L^{2}}}
\providecommand{\lnm}[1]{\left\Vert#1\right\Vert_{L^{\infty}}}
\providecommand{\lnmm}[1]{\left\Vert#1\right\Vert_{L^{\infty}_{\varrho,\vartheta}}}
\providecommand{\um}[1]{\left\Vert#1\right\Vert_{L^2_{\nu}}}

\providecommand{\pnms}[3]{\left\vert#1\right\vert_{L^{#2}_{#3}}}
\providecommand{\tnms}[2]{\left\vert#1\right\vert_{L^{2}_{#2}}}
\providecommand{\lnms}[2]{\left\vert#1\right\vert_{L^{\infty}_{#2}}}
\providecommand{\lnmms}[2]{\left\vert#1\right\vert_{L^{\infty}_{#2,\varrho,\vartheta}}}

\def\ud{\mathrm{d}}

\def\p{\partial}
\def\ls{\lesssim}
\def\gs{\gtrsim}

\def\rt{\rightarrow}
\def\r{\mathbb{R}}
\def\no{\nonumber}
\def\ue{\mathrm{e}}

\def\ds{\displaystyle}

\def\oo{o(1)}
\def\oot{o_T}

\def\id{\textbf{1}}
\def\od{\textbf{0}}

\def\N{r}
\def\NN{s}
\def\P{p}

\def\fs{\mathfrak{F}}
\def\f{f}
\def\fb{f^{B}}
\def\pe{\mathfrak{P}_{\gamma}}
\def\pp{\mathcal{P}_{\gamma}}

\def\e{\varepsilon}
\def\al{\alpha}
\def\s{\mathbb{S}}

\def\vx{x}
\def\vv{v}
\def\vuu{{\mathfrak{u}}}
\def\vvv{{\mathfrak{v}}}
\def\nx{\nabla_{x}}
\def\dx{\Delta_x}
\def\d{\delta}
\def\vn{n}

\def\k{\kappa}
\def\re{R}
\def\ss{S}

\def\nk{\mathcal{N}}
\def\nnk{\mathcal{N}^{\perp}}

\def\sb{\overline{\ss}}
\def\sp{\ss_6}
\def\ch{\overline{\chi}}

\def\ff{\mathfrak{F}_{\text{a}}}
\def\ffe{\widetilde{\mathfrak{F}}_{\text{a}}}

\def\m{\mu}
\def\mh{\m^{\frac{1}{2}}}
\def\mhh{\m^{-\frac{1}{2}}}
\def\mb{\mu_{w}}
\def\mbh{\mb^{\frac{1}{2}}}

\def\ms{M_{w}}
\def\mss{m_w}

\def\rq{\rho}
\def\uq{u}
\def\tq{T}

\def\tb{T_{w}}

\def\lc{\mathcal{L}}
\def\llc{\lc^1}
\def\li{\lc^{-1}}
\def\pk{\textbf{P}}
\def\bpk{\overline{\textbf{P}}}
\def\ik{\textbf{I}}

\def\a{\mathscr{A}}
\def\ab{\overline{\mathscr{A}}}
\def\bb{\textbf{b}}
\def\bd{\textbf{d}}
\def\be{\textbf{e}}

\def\b{\mathscr{B}}
\def\bbb{\overline\b}

\def\vo{\omega}

\def\va{v_{\eta}}
\def\vb{v_{\phi}}
\def\vc{v_{\psi}}
\def\vr{\mathbf{r}}
\def\vt{\varsigma}
\def\kk{\kappa}
\def\mn{\mathfrak{n}}

\def\tm{T_M}
\def\mm{\mu_M}
\def\mmh{\mm^{\frac{1}{2}}}
\def\mmhh{\mm^{-\frac{1}{2}}}
\def\rem{\re_M}

\def\mmss{m_{M,w}}

\def\blf{\Phi}
\def\blff{\widetilde{\Phi}}

\def\pl{L}

\def\as{\abs{\nabla\tb}_{W^{3,\infty}}}

\def\bu{\overline{u}}

\begin{document}

\title{Ghost Effect from Boltzmann Theory: Expansion with Remainder}

\author[R. Esposito]{Raffaele Esposito}
\address[R. Esposito]{
   \newline\indent International Research Center M\&MOCS - Universita' dell'Aquila}
\email{raff.esposito@gmail.com}

\author[Y. Guo]{Yan Guo}
\address[Y. Guo]{
   \newline\indent Division of Applied Mathematics, Brown University}
\email{yan\_guo@brown.edu}
\thanks{Y. Guo was supported by NSF Grant DMS-2106650.}

\author[R. Marra]{Rossana Marra}
\address[R. Marra]{
   \newline\indent Dipartimento di Fisica and Unit\`a INFN, Universit\`a di Roma Tor Vergata}
\email{marra@roma2.infn.it}
\thanks{R. Marra is supported by INFN}

\author[L. Wu]{Lei Wu}
\address[L. Wu]{
   \newline\indent Department of Mathematics, Lehigh University}
\email{lew218@lehigh.edu}
\thanks{L. Wu was supported by NSF Grant DMS-2104775.}

\date{}



\maketitle

\makeatletter
\renewcommand \theequation {%
0.%
\ifnum \c@section>\z@ \@arabic\c@section.%
\fi
\@arabic\c@equation} \@addtoreset{equation}{section}
\makeatother

\begin{abstract}
Consider the limit $\e\rt0$ of the steady Boltzmann problem
\begin{align}\label{large system-}
\vv\cdot\nx \fs=\e^{-1}Q[\fs,\fs],\quad \fs\big|_{\vv\cdot\vn<0}=\ms\displaystyle\int_{\vv'\cdot\vn>0}
    \fs(\vv')\abs{\vv'\cdot\vn}\ud{\vv'},
\end{align}
where $\ds\ms(\vx_0,\vv):=\frac{1}{2\pi\big(\tb(\vx_0)\big)^2}
\exp\bigg(-\frac{\abs{\vv}^2}{2\tb(\vx_0)}\bigg)$ for $x_0\in\p\Omega$ is the wall Maxwellian in the diffuse-reflection boundary condition. We normalize 
\begin{align}
    \tb=1+O\big(\abs{\nabla\tb}_{L^{\infty}}\big).
\end{align}
In the case of $\abs{\nabla\tb}=O(\e)$, the Hilbert expansion confirms $\ds\fs\approx (2\pi)^{-\frac{3}{2}}\ue^{-\frac{\abs{v}^2}{2}}+\e(2\pi)^{-\frac{3}{4}}\ue^{-\frac{\abs{v}^2}{4}}\bigg(\rq_1+\tq_1\frac{\abs{v}^2-3}{2}\bigg)$ where $(2\pi)^{-\frac{3}{2}}\ue^{-\frac{\abs{v}^2}{2}}$ is a global Maxwellian and $(\rq_1,\tq_1)$ satisfies the celebrated Fourier law
\begin{align}
    \Delta_x\tq_1=0.
\end{align}
In the natural case of $\abs{\nabla\tb}=O(1)$, for any constant $P>0$, the Hilbert expansion leads to
\begin{align}\label{expansion}
    \fs\approx \m+\e\bigg\{\m\bigg(\rq_1+\uq_1\cdot\vv+\tq_1\frac{\abs{v}^2-3T}{2}\bigg)-\mh\left(\a\cdot\frac{\nx\tq}{2\tq^2}\right)\bigg\}
\end{align}
where $\ds\m(\vx,\vv):=\frac{\rq(\vx)}{\big(2\pi\tq(\vx)\big)^{\frac{3}{2}}}
\exp\bigg(-\frac{\abs{\vv}^2}{2\tq(\vx)}\bigg)$,
and $(\rq,\uq_1,\tq)$ is determined by a Navier-Stokes-Fourier system with ``ghost" effect 
\begin{align}\label{fluid system-}
\left\{
\begin{array}{rcl}
    P&=&\rq\tq,\\\rule{0ex}{1.2em}
    \rq(\uq_1\cdot\nx\uq_1)+\nx \mathfrak{p}&=&\nx\cdot\left(\tau^{(1)}-\tau^{(2)}\right),\\\rule{0ex}{1.2em}
    \nx\cdot(\rq\uq_1)&=&0,\\\rule{0ex}{1.8em}
    \nx\cdot\left(\k\dfrac{\nx\tq}{2\tq^2}\right)&=&5P(\nx\cdot\uq_1),
    \end{array}
    \right.
\end{align}
with the boundary condition 
\begin{align}\label{boundary condition}
    \tq\Big|_{\p\Omega}=\tb,\quad \uq_1\Big|_{\p\Omega}:=\big(u_{1,\iota_1},u_{1,\iota_2},u_{1,n}\big)\Big|_{\p\Omega}=\Big(\beta_0\p_{\iota_1}\tb,\beta_0\p_{\iota_2}\tb,0\Big).
\end{align}
Here $\k[\tq]>0$ is the heat conductivity, $(\iota_1,\iota_2)$ are two tangential variables and $n$ is the normal variable, $\beta_0=\beta_0[\tb]$ is a function of $\tb$,
$\tau^{(1)}:=\lambda\left(\nx\uq_1+(\nx\uq_1)^t-\frac{2}{3}(\nx\cdot\uq_1)\id\right)$ and $\tau^{(2)}:=\frac{\lambda^2}{P}\Big(K_1\big(\nx^2\tq-\frac{1}{3}\dx\tq\id\big)+\frac{K_2}{\tq}\big(\nx\tq\otimes\nx\tq-\frac{1}{3}\abs{\nx\tq}^2\id\big)\Big)$ for some smooth function $\lambda[T]>0$, the viscosity coefficient, and positive constants $K_1$ and $K_2$.
Tangential temperature variation creates
non-zero first-order velocity $\uq_1$ at the boundary \eqref{boundary condition}, which plays a
surprising ``ghost" effect \cite{Sone2002, Sone2007} in determining zeroth-order density
and temperature field $(\rq,\tq)$ in \eqref{fluid system-}. Such a ghost effect cannot be predicted by the classical fluid theory, while it has been an intriguing outstanding mathematical problem
to justify \eqref{fluid system-} from \eqref{large system-} due to fundamental analytical challenges. 
The goal of this paper is to construct $\fs$ in the form of 
\begin{align}\label{aa 08}
    \fs(x,v)=&\m+\mh\Big(\e\f_1+\e^2\f_2\Big)+\mbh\Big(\e\fb_1\Big)+\e^{\al}\mh\re,
\end{align}
for interior solutions $\f_1$, $\f_2$ and boundary layer $\fb_1$, where $\mb$ is $\m$ computed for $\tq=\tb$, and derive equation
for the remainder $\re$ with some constant $\al\geq1$. To prove the validity of the expansion suitable bounds on $\re$ are needed, which are provided in the companion paper \cite{AA023}.
\end{abstract}

\pagestyle{myheadings} \thispagestyle{plain} \markboth{ESPOSITO AND GUO AND MARRA AND WU}{GHOST EFFECT FROM BOLTZMANN THEORY}

\setcounter{tocdepth}{1}
\tableofcontents


\makeatletter
\renewcommand \theequation {%
\ifnum \c@section>\z@ \@arabic\c@section.%
\fi\@arabic\c@equation} \@addtoreset{equation}{section}
\makeatother

\section{Introduction}

The diffusive hydrodynamic limit of the Boltzmann equation in the low Mach number regime is described by the incompressible Navier-Stokes-Fourier equations under the extra assumption that the initial density and temperature profiles differ from constants at most for terms of the order of the Knudsen number. Such behavior has been proved in several papers and an overview is provided in \cite{Saint-Raymond2009} and \cite{Esposito.Pulvirenti2010}, to which we refer for a partial list of references on the subject. We also stress that a similar result can be obtained starting from the compressible Navier-Stokes equations, which converge, in the low Mach number limit, to the solutions of the incompressible Navier-Stokes equations \cite{Klainerman.Majda1981}.

When the density and temperature do not satisfy the above mentioned assumptions, the limiting behavior of the Boltzmann equation deviates from the Navier-Stokes-Fourier equations. Such a discrepancy, called ``ghost effect'' \cite{Sone2007}, shows up in the macroscopic equations with the presence of some extra terms reminiscent of the limiting procedure such as some heat flow induced by the vanishingly small velocity field. Thus they are genuine kinetic effects which would be never detected in the standard hydrodynamic equations. Y. Sone has given the suggestive name of ``ghost effects'' to such phenomena. The meaning of the name is that the velocity field $u_1$ acts like a ghost since it appears at order $\e$ in the expansion and still affects $\rho$ and $T$ at order $1$. In \cite{Levermore.Sun.Trivisa2012} the local well-posedness of the time dependent equations is proven.

In this paper we confine our analysis to the stationary Boltzmann equation for a rarefied gas in a bounded  domain with diffuse-reflection boundary data describing a non-homogeneous wall temperature with a gradient of order $1$. In this situation the gradient of temperature along the boundary wall produces a flow called in literature thermal creep. For relevant physical background and discussion, we refer to \cite{Sone1966}.

We give a formal derivation of such new equations when the Mach number, proportional to the Knudsen number $\e$, goes to $0$, and prove their well-posedness. In the companion paper \cite{AA023} we study the much more involved problem of the rigorous proof of such a derivation. Here we construct the formal solution by a truncated expansion in $\e$ plus a remainder, both in the interior and in a boundary layer of size $\e$. In view of the control of the remainder, we carefully prepare the expansion by truncating at the second order in $\e$ in the bulk and at the first order in the boundary layer. Then a matching procedure allows to determine the boundary conditions for the limiting equations.

The explicit form of the equations for $(\rq, \uq_1, T, \mathfrak{p})$ is given in \eqref{fluid system-}. The main difference between these equations and the incompressible ones is that $\nx\cdot\uq_1$ is not anymore zero but is related to the gradient of the temperature. This is the analog of the constraint $\nx\cdot\uq_1=0$ in the incompressible Navier-Stokes equations and is compensated by the Lagrangian multiplier $\mathfrak{p}$ in the equation for $u_1$. Moreover, in the equation for $u_1$ there are the usual Navier-Stokes terms involving $\uq_1$ and also a term $\tau^{(2)}$ depending on the first and second gradient of the temperature. It is exactly this term that cannot be obtained from the compressible Navier-Stokes equation.
The relevance of these equations, as also noted by Bobylev \cite{Bobylev1995}, is that they cannot be derived from the compressible Navier-Stokes equations. 
Let us notice that the particular solution corresponding to homogeneous initial condition for density and temperature is also solution of the incompressible Navier-Stokes equations. 

We give also the proof of the existence of the solution to \eqref{fluid system-} under the assumption of small temperature gradient.
The main difficulty in getting a rigorous proof of the hydrodynamic limit is the control of the remainder. This is achieved in \cite{AA023}.

Before stating the main results, we briefly introduce the history of the study of the ghost effect.
\cite{Sone1972} and \cite{Kogan1958}, \cite{Kogan.Galkin.Fridlender1976} pointed out the new thermal effects in stationary situations. In \cite{Masi.Esposito.Lebowitz1989} the equations from the Boltzmann equations in the time dependent case were formally derived, but without computing the transport coefficients. 
These equations were then discussed by Bobylev \cite{Bobylev1995}, who analyzed the behavior of the solutions in particular situations. He also showed that the thermodynamic entropy decreases in time. Finally, Sone and the Kyoto group exploited many other kinds of ghost effects in many papers \cite{Sone.Aoki.Takata.Sugimoto.Bobylev1996, Takata.Aoki2001}, both analytically and numerically and gave computations of the transport coefficients for the hard sphere case and for Maxwellian molecules. A detailed analysis can be found in \cite{Sone2002} and \cite{Sone2007} and references therein. Rigorous results in deriving the equations where obtained only in one-dimensional stationary cases \cite{Brull2008}, \cite{Brull2008(=)} and \cite{Arkeryd.Esposito.Marra.Nouri2011}. There are no rigorous results in the time dependent case, not even on the torus, but for \cite{Huang.Wang.Wang.Yang2016} where the Korteweg theory is derived from the one-dimensional Boltzmann equation on the infinite line. We also refer to \cite{Huang.Tan2017, Jiang.Levermore2011} and the references therein.

\subsection{Formulation of the Problem}

We consider the stationary Boltzmann equation in a bounded three-dimensional $C^3$ domain $\Omega\ni\vx=(x_1,x_2,x_3)$
with velocity $\vv=(v_1,v_2,v_3)\in\r^3$. The density function $\fs(\vx,\vv)$ satisfies 
\begin{align}\label{large system}
\left\{
\begin{array}{l}
\vv\cdot\nx \fs=\e^{-1}Q[\fs,\fs]\ \ \text{in}\ \
\Omega\times\r^3,\\\rule{0ex}{1.5em} \fs(\vx_0,\vv)=\pe[\fs] \ \ \text{for}\ \ \vx_0\in\p\Omega\ \ \text{and}\ \ \vv\cdot\vn(\vx_0)<0.
\end{array}
\right.
\end{align}
Here $Q$ is the hard-sphere collision operator
\begin{align}
Q[F,G]:=&\frac{1}{2}\int_{\r^3}\int_{\s^2}q(\vo,\abs{\vuu-\vv})\Big(F(\vuu_{\ast})G(\vv_{\ast})-F(\vuu)G(\vv)\Big)\ud{\vo}\ud{\vuu},
\end{align}
with $\vuu_{\ast}:=\vuu+\vo\big((\vv-\vuu)\cdot\vo\big)$, $\vv_{\ast}:=\vv-\vo\big((\vv-\vuu)\cdot\vo\big)$, and the hard-sphere collision kernel $q(\vo,\abs{\vuu-\vv}):=q_0\abs{\vo\cdot(\vv-\vuu)}$ for a positive constant $q_0$.

In the diffuse-reflection boundary condition
\begin{align}
    \pe[\fs]:=\ms(\vx_0,\vv)\displaystyle\int_{\vv'\cdot\vn(\vx_0)>0}
    \fs(\vx_0,\vv')\abs{\vv'\cdot\vn(\vx_0)}\ud{\vv'},
\end{align}
$\vn(\vx_0)$ is the unit outward normal vector at $\vx_0$, and the
Knudsen number $\e$ satisfies $0<\e\ll 1$. The wall Maxwellian
\begin{align}
\ms(\vx_0,\vv):=\frac{1}{2\pi\big(\tb(\vx_0)\big)^2}
\exp\bigg(-\frac{\abs{\vv}^2}{2\tb(\vx_0)}\bigg),
\end{align}
for any $\tb(x_0)>0$ satisfies
\begin{align}
\int_{\vv\cdot\vn(\vx_0)>0}\ms(\vx_0,\vv)\abs{\vv\cdot\vn(\vx_0)}\ud{\vv}=1.
\end{align}
The boundary condition in \eqref{large system} implies that the total max flux across the boundary is zero.

\subsection{Notation and Convention}

Based on the flow direction, we can divide the boundary $\gamma:=\big\{(\vx_0,\vv):\ \vx_0\in\p\Omega,\vv\in\r^3\big\}$ into the incoming boundary $\gamma_-$, the outgoing boundary $\gamma_+$, and the grazing set $\gamma_0$ based on the sign of $\vv\cdot\vn(\vx_0)$. In particular, the boundary condition of \eqref{large system} is only given on $\gamma_{-}$.

Denote the bulk and boundary norms
\begin{align}
    \pnm{f}{r}:=\left(\iint_{\Omega\times\r^3}\abs{f(x,v)}^r\ud v\ud x\right)^{\frac{1}{r}},\quad \pnms{f}{r}{\gamma_{\pm}}:=\left(\int_{\gamma_{\pm}}\abs{f(x,v)}^r\abs{v\cdot n}\ud v\ud x\right)^{\frac{1}{r}}.
\end{align}
Define the weighted $L^{\infty}$ norms for $\tm>0$ defined in \eqref{def:tm}, $0\leq\varrho<\dfrac{1}{2}$ and $\vartheta\geq0$
\begin{align}
    \lnmm{f}:=\esssup_{(x,v)\in\Omega\times\r^3}\bigg(\br{v}^{\vartheta}\ue^{\varrho\frac{\abs{v}^2}{2\tm}}\abs{f(x,v)}\bigg),\quad
    \lnmms{f}{\gamma_{\pm}}:=\esssup_{(x,v)\in\gamma_{\pm}}\bigg(\br{v}^{\vartheta}\ue^{\varrho\frac{\abs{v}^2}{2\tm}}\abs{f(x,v)}\bigg).
\end{align}
Denote the $\nu$-norm
\begin{align}
    \um{f}:=\left(\iint_{\Omega\times\r^3}\nu(x,v)\abs{f(x,v)}^2\ud v\ud x\right)^{\frac{1}{2}}.
\end{align}
Let $\nm{\cdot}_{W^{k,p}}$ denote the usual Sobolev norm for $x\in\Omega$ and  $\abs{\cdot}_{W^{k,p}}$ for $x\in\p\Omega$. Let $\nm{\cdot}_{W^{k,p}L^q}$ denote $W^{k,p}$ norm for $x\in\Omega$ and $L^q$ norm for $v\in\r^3$. The similar notation also applies when we replace $L^q$ by $L^{\infty}_{\varrho,\vartheta}$ or $L^q_{\gamma}$.

Define the quantities (where $\lc$ is defined in \eqref{att 11})
\begin{align}\label{final 22}
    &\ab:=v\cdot\left(\abs{v}^2-5T\right)\mh\in\r^3,\quad \a:=\li\left[\ab\right]\in\r^3,\\
    &\bbb=\bigg(v\otimes v-\frac{\abs{v}^2}{3}\mathbf{1}\bigg)\mh\in\r^{3\times3},\quad \b=\lc^{-1}\left[\bbb\right]\in\r^{3\times3},\\
    &\k\id:=\int_{\r^3}\left(\a\otimes\ab\right)\ud v,\quad 
    \lambda:=\frac{1}{\tq}\int_{\r^3}\b_{ij}\bbb_{ij}\ \ \text{for}\ \ i\neq j.\label{final 23}
\end{align}

Throughout this paper, $C>0$ denotes a constant that only depends on
the domain $\Omega$, but does not depend on the data or $\e$. It is
referred as universal and can change from one inequality to another.
When we write $C(z)$, it means a certain positive constant depending
on the quantity $z$. We write $a\ls b$ to denote $a\leq Cb$ and $a\gs b$ to denote $a\geq Cb$.

In this paper, we will use $\oo$ to denote a sufficiently small constant independent of the data. Also, let $\oot$ be a small constant depending on $\tb$ satisfying
\begin{align}\label{def:oot}
    \oot=\oo\rt0\ \ \text{as}\ \ \as\rt0.
\end{align}
In principle, while $\oot$ is determined by $\nabla\tb$ a priori, we are free to choose $\oo$ depending on the estimate.

\subsection{Main Theorem}

Throughout this paper, we assume that 
\begin{align}\label{assumption:boundary}
    \as=\oo.
\end{align}

\begin{theorem}\label{thm:main}
    Under the assumption \eqref{assumption:boundary}, for any given $P>0$, there exists a unique solution $(\rq,\uq_1,\tq; \mathfrak{p})$ (where $\mathfrak{p}$ has zero average) to the ghost-effect equation \eqref{fluid system-} and \eqref{boundary condition} satisfying for any $\NN\in[2,\infty)$
    \begin{align}
    \nm{u_1}_{W^{3,\NN}}+\nm{\mathfrak{p}}_{W^{2,\NN}}+\nm{T-1}_{W^{4,\NN}}\ls\oot.
    \end{align}
    Also, we can construct $\f_1$, $\f_2$ and $\fb_1$ as in \eqref{extra 34}, \eqref{extra 35}, \eqref{extra 15} such that 
    \begin{align}
    \nm{f_1}_{W^{3,\NN}L^{\infty}_{\varrho,\vartheta}}+\abs{f_1}_{W^{3-\frac{1}{\NN},\NN}L^{\infty}_{\varrho,\vartheta}}&\ls\oot,\\
    \nm{f_2}_{W^{2,\NN}L^{\infty}_{\varrho,\vartheta}}+\abs{f_2}_{W^{2-\frac{1}{\NN},\NN}L^{\infty}_{\varrho,\vartheta}}&\ls\oot,
    \end{align}
and for some $K_0>0$ and any $0<\N\leq 3$
\begin{align}
    \lnmm{\ue^{K_0\eta}\fb_1}+\lnmm{\ue^{K_0\eta}\frac{\p^\N\fb_1}{\p\iota_1^\N}}+\lnmm{\ue^{K_0\eta}\frac{\p^\N\fb_1}{\p\iota_2^\N}}&\ls\oot.
\end{align}
\end{theorem}

\section{Asymptotic Analysis}\label{sec:asymptotic}

In this section we construct a solution to \eqref{large system} by a truncated expansion in $\e$ and determine the ghost effect equation in terms of the first terms of the expansion.

We seek a solution in the form
{\begin{align}\label{aa 08'}
    \fs(x,v)=&\f+\fb+\e^{\al}\mh\re\\
    =&\m+\mh\Big(\e\f_1+\e^2\f_2\Big)+\mbh\Big(\e\fb_1\Big)+\e^{\al}\mh\re,\no
\end{align}}
where $f$ is the interior solution
\begin{align}\label{expand 1}
\f(x,v):= \m(x,v)+\mh(x,v)\Big(\e\f_1(x,v)+\e^2\f_2(x,v)\Big),
\end{align}
and $\fb$ is the boundary layer term
\begin{align}\label{expand 2}
\fb(x,v):= \mbh(x_0,v)\Big(\e\fb_1(x,v)\Big).
\end{align}
Here $\re(x,v)$ is the remainder, $\m(x,v)$ denotes a local Maxwellian which will be specified below and $\mb(x_0,v)=\m(x_0,v)$ the boundary Maxwellian. The parameter $\al\geq1$, will be equal to $1$ in the companion paper \cite{AA023}.

We start to determine the first terms of the expansion. Inserting \eqref{expand 1} into \eqref{large system}, at the lowest order of $\e$, we have
\begin{align}
\textbf{Order $0$:}&\quad-Q\big[\m,\m\big]=0.\ &
\end{align}
This equation guarantees that $\m$ is a local Maxwellian. Denote
\begin{align}
\m(\vx,\vv):=\frac{\rq(\vx)}{\big(2\pi\tq(\vx)\big)^{\frac{3}{2}}}
\exp\left(-\frac{\abs{\vv}^2}{2\tq(\vx)}\right),
\end{align}
where $\rq(x)>0$ and $\tq(x)>0$ will be determined later in terms of the solutions of the ghost equations. Notice that this local Maxwellian does not contain the velocity field since we are assuming the Mach number of order $\e$.

\paragraph{\underline{Linearized Boltzmann Operator}}\label{sec:boltzmann}

Define the symmetrized version of $Q$
\begin{align}\label{extra 21}
Q^{\ast}[F,G]:=&\frac{1}{2}\iint_{\r^3\times\s^2}q(\vo,\abs{\vuu-\vv})\Big(F(\vuu_{\ast})G(\vv_{\ast})+F(\vv_{\ast})G(\vuu_{\ast})-F(\vuu)G(\vv)-F(\vv)G(\vuu)\Big)\ud{\vo}\ud{\vuu}.
\end{align}
Clearly, $Q[F,F]=Q^{\ast}[F,F]$.
Denote the linearized Boltzmann operator $\lc$
\begin{align}\label{att 11}
\lc[f]:=&-2\m^{-\frac{1}{2}}Q^{\ast}\big[\m,\m^{\frac{1}{2}}f\big]:=\nu(\vv)f-K[f],\end{align}
where
for some kernels $k(\vuu,\vv)$ (see \cite{Glassey1996, Cercignani.Illner.Pulvirenti1994}),
\begin{align}
\nu(\vv)=&\int_{\r^3}\int_{\s^2}q(\vo,\abs{\vuu-\vv})\m(\vuu)\ud{\vo}\ud{\vuu},\\
K[f](\vv)
=&\int_{\r^3}\int_{\s^2}q(\vo,\abs{\vuu-\vv})\m^{\frac{1}{2}}(\vuu)\bigg(\m^{\frac{1}{2}}(\vv_{\ast})f(\vuu_{\ast})
+\m^{\frac{1}{2}}(\vuu_{\ast})f(\vv_{\ast})\bigg)\ud{\vo}\ud{\vuu}\\
&-\m^{\frac{1}{2}}(\vv)\int_{\r^3}\int_{\s^2}q(\vo,\abs{\vuu-\vv})\m^{\frac{1}{2}}(\vuu)f(\vuu)\ud{\vo}\ud{\vuu}.\no
\end{align}
Note that $\lc$ is self-adjoint in $L^2_{\nu}(\r^3)$. Also, the null space $\nk$ of $\lc$ is a five-dimensional space spanned by the orthogonal basis
\begin{align}\label{att 32}
\mh\bigg\{1,\vv,\left(\abs{\vv}^2-3\tq\right)\bigg\}.
\end{align}
Denote $\nnk$ the orthogonal complement of $\nk$ in $L^2(\r^3)$, and $\li: \nnk\rt\nnk$ the quasi-inverse of $\lc$.
Define the kernel operator $\pk$
as the orthogonal projection onto the null space $\nk$ of $\lc$, and the non-kernel operator $\ik-\pk$. 
Also, denote the nonlinear Boltzmann operator $\Gamma$ as
\begin{align}
\Gamma[f,g]:=\mhh Q^{\ast}\left[\mh f,\mh g\right]\in\nnk.
\end{align}

\subsection{Derivation of Interior Solution}

Further inserting \eqref{expand 1} into \eqref{large system}, we have
\begin{align}
\textbf{Order $1$:}&\quad\vv\cdot\nx\m-2Q^{\ast}\left[\m,\mh\f_1\right]=0,\label{expand 3}\\
\textbf{Order $\e$:}&\quad\vv\cdot\nx\left(\mh\f_1\right)-2Q^{\ast}\left[\m,\mh\f_2\right]-Q^{\ast}\left[\mh\f_1,\mh\f_1\right]=0.\label{expand 4}
\end{align}
Inspired by the continuation of the expansion, we also require an additional condition that
\begin{align}\label{expand 12}
\textbf{Order $\e^2$:}&\quad\mhh\left(\vv\cdot\nx\left(\mh\f_2\right)\right)\perp v\mh.
\end{align} 
Note that we stop the bulk expansion at order $\e^2$, so we do not need the orthogonality with $\mh$ and $\abs{v}^2\mh$.

\subsubsection{\underline{Equation \eqref{expand 3}}}

\begin{lemma}
    Equation \eqref{expand 3} is equivalent to
    \begin{align}\label{extra 01}
    \nx P=\nx(\rq\tq)=0,
    \end{align}
    and for some $\rq_1(x),\uq_1(x),\tq_1(x)$,
    \begin{align}\label{expand 8}
    \f_1=&-\a\cdot\frac{\nx\tq}{2\tq^2}+\mh\bigg(\frac{\rq_1}{\rq}+\frac{\uq_1\cdot\vv}{\tq}+\frac{\tq_1(\abs{\vv}^2-3\tq)}{2\tq^2}\bigg).
    \end{align}
\end{lemma}

\begin{proof}
Equation \eqref{expand 3} can be rewritten as 
\begin{align}\label{expand 7}
\mhh\big(\vv\cdot\nx\m\big)=-\lc[\f_1].
\end{align}
Then, by the orthogonality of $\lc$ to  $\nk$, to satisfy \eqref{expand 7} we must have
\begin{align}\label{expand 7'}
\int_{\r^3}\big(\vv\cdot\nx\m\big)\ud\vv=0,\quad\int_{\r^3}\vv\big(\vv\cdot\nx\m\big)\ud\vv=\od,\quad\int_{\r^3}\abs{\vv}^2\big(\vv\cdot\nx\m\big)\ud\vv=0.
\end{align}
Note that
\begin{align}\label{expand 5}
v\cdot\nx\m=\m\left(\vv\cdot\frac{\nx\rq}{\rq}+\vv\cdot\frac{\nx\tq(\abs{\vv}^2-3\tq)}{2\tq^2}\right).
\end{align}
Then the first and third conditions in \eqref{expand 7'} are satisfied by oddness. The second condition in \eqref{expand 7'} can be rewritten in the component form for $i\in\{1,2,3\}$ and summation over $j\in\{1,2,3\}$
\begin{align}
&\int_{\r^3}v_iv_j\m\left(\frac{\p_j\rq}{\rq}+\frac{\p_j\tq(\abs{\vv}^2-3\tq)}{2\tq^2}\right)\ud\vv=\int_{\r^3}\d_{ij}\frac{\abs{v}^2}{3}\m\left(\frac{\p_j\rq}{\rq}+\frac{\p_j\tq(\abs{\vv}^2-3\tq)}{2\tq^2}\right)\ud\vv \label
{rhoT}\\=&\d_{ij}\bigg(\rq\tq\cdot\frac{\p_j\rq}{\rq}
+5\rq\tq^2\cdot\frac{\p_j\tq}{2\tq^2}-\rq\tq\cdot\frac{3\p_j\tq}{2\tq}\bigg)=\d_{ij}\Big(\tq\p_j\rq+\rq\p_j\tq\Big)=\d_{ij}\p_j(\rq\tq)=0.\no
\end{align}
Hence equation \eqref{rhoT} is actually \eqref{extra 01}.

Since $\tq\nx\rq+\rq\nx\tq=0$, we deduce $\dfrac{\nx\rq}{\rq}=-\dfrac{\nx\tq}{\tq}$. Thus, inserting this into \eqref{expand 5}, we have
\begin{align}\label{expand 6}
v\cdot\nx\m=\m\big(\vv\cdot\nx\tq\big)\frac{\abs{\vv}^2-5\tq}{2\tq^2}.
\end{align}
Considering \eqref{expand 7} and \eqref{expand 6}, we know
\begin{align}
\lc[f_1]=-\ab\cdot\frac{\nx\tq}{2\tq^2},
\end{align}
and \eqref{expand 8} holds.
\end{proof}

\subsubsection{\underline{Equation \eqref{expand 4}}}

\begin{lemma}
    Equation \eqref{expand 4} is equivalent to
    \begin{align}
    &\nx\cdot(\rq\uq_1)=0,\label{extra 02}\\
    &\nx P_1=\nx(\tq\rq_1+\rq\tq_1)=0,\label{aa 23}\\
    &5P(\nx\cdot\uq_1)=\nx\cdot\left(\k\frac{\nx\tq}{2\tq^2}\right),\label{extra 03}
    \end{align}
    and for some $\rq_2(x),\uq_2(x),\tq_2(x)$,
    \begin{align}\label{expand 14}
    \f_2=-\li\Big[\mhh\vv\cdot\nx(\mh\f_1)\Big]+\li\Big[\Gamma[\f_1,\f_1]\Big]
    +\mh\bigg(\frac{\rq_2}{\rq}+\frac{\uq_2\cdot\vv}{\tq}+\frac{\tq_2(\abs{\vv}^2-3\tq)}{2\tq^2}\bigg).
    \end{align}
\end{lemma}

\begin{proof}
Since the $Q^*$ terms in \eqref{expand 4} are orthogonal to $\nk$, we must have
\begin{align}\label{expand 18}
\int_{\r^3}\Big(\vv\cdot\nx(\mh\f_1)\Big)\ud\vv=0,\quad\int_{\r^3}\vv\Big(\vv\cdot\nx(\mh\f_1)\Big)\ud\vv=0,\quad\int_{\r^3}\abs{\vv}^2\Big(\vv\cdot\nx(\mh\f_1)\Big)\ud\vv=0.
\end{align}
Using \eqref{expand 8}, the first condition in \eqref{expand 18} can be rewritten as
\begin{align}\label{expand 19}
\nx\cdot\left(-\int_{\r^3}\vv\mh\a\cdot\frac{\nx\tq}{2\tq^2}\ud\vv
+\int_{\r^3}\vv\m\bigg(\frac{\rq_1}{\rq}+\frac{\uq_1\cdot\vv}{\tq}+\frac{\tq_1(\abs{\vv}^2-3\tq)}{2\tq^2}\bigg)\ud\vv\right)=0.
\end{align}
Since $\a$ is orthogonal to $\nk$, the first term in \eqref{expand 19} vanishes. Due to oddness, the $\rq_1$ and $\tq_1$ terms in \eqref{expand 19} vanish. Hence, we are left with \eqref{extra 02}.

Similarly, the second condition in \eqref{expand 18} can be rewritten as
\begin{align}\label{expand 20}
\nx\cdot\left(-\int_{\r^3}\vv\otimes\vv\mh\a\cdot\frac{\nx\tq}{2\tq^2}\ud\vv
+\int_{\r^3}\vv\otimes\vv\m\bigg(\frac{\rq_1}{\rq}+\frac{\uq_1\cdot\vv}{\tq}+\frac{\tq_1(\abs{\vv}^2-3\tq)}{2\tq^2}\bigg)\ud\vv\right)=0.
\end{align}
Due to the oddness of $\a$, the first term in \eqref{expand 20} vanishes. For the same reason, the $\uq_1$ term in \eqref{expand 20} also vanishes. Thus we are left with \eqref{aa 23}.

Finally, the third condition in \eqref{expand 18} can be rewritten as
\begin{align}\label{expand 11}
\nx\cdot\left(-\int_{\r^3}\vv\abs{\vv}^2\mh\a\cdot\frac{\nx\tq}{2\tq^2}\ud\vv
+\int_{\r^3}\vv\abs{\vv}^2\m\bigg(\frac{\rq_1}{\rq}+\frac{\uq_1\cdot\vv}{\tq}+\frac{\tq_1(\abs{\vv}^2-3\tq)}{2\tq^2}\bigg)\ud\vv\right)=0.
\end{align}
Using the orthogonality of $\a$ to $\nk$, we know
\begin{align}\label{expand 9}
\int_{\r^3}\vv\abs{\vv}^2\mh\a\cdot\frac{\nx\tq}{2\tq^2}\ud\vv=\int_{\r^3}\ab\a\cdot\frac{\nx\tq}{2\tq^2}\ud\vv=\k\frac{\nx\tq}{2\tq^2},
\end{align}
where $\k$ is defined in \eqref{final 23}.

Due to oddness, the $\rq_1$ and $\tq_1$ terms in \eqref{expand 11} vanish, so the $\uq_1$ term in \eqref{expand 11} can be computed
\begin{align}\label{expand 10}
\int_{\r^3}\vv\abs{\vv}^2\m\frac{\uq_1\cdot\vv}{\tq}\ud\vv=5\rq\tq\uq_1=5P\uq_1.
\end{align}
Hence, \eqref{expand 11} becomes
\begin{align}
\nx\cdot\bigg(-\k\frac{\nx\tq}{2\tq^2}+5P\uq_1\bigg)=0,
\end{align}
which is equivalent to \eqref{extra 03}.

\eqref{expand 4} can be rewritten as
\begin{align}\label{expand 13}
\mhh\vv\cdot\nx(\mh\f_1)-\Gamma[\f_1,\f_1]=-\lc[\f_2],
\end{align}
and thus \eqref{expand 14} holds.
\end{proof}

\subsubsection{\underline{Equation \eqref{expand 12}}}

\begin{lemma}
    We have the identity
    \begin{align}\label{extra 06}
        \int_{\r^3}\b\Gamma\left[(u_1\cdot v)\mh,\a\right]=-\int_{\r^3}\b\left(\frac{u_1\cdot v}{2}\right)\ab+T\int_{\r^3}\b\big(u_1\cdot\bbb\big).
    \end{align}
\end{lemma}

\begin{proof}
We follow the idea in \cite{Bardos.Levermore.Ukai.Yang2008}.
Denote the translated quantities
\begin{align}
    \m_s(\vx,\vv):=\frac{\rq(\vx)}{\big(2\pi\tq(\vx)\big)^{\frac{3}{2}}}
    \exp\left(-\frac{\abs{\vv-su_1}^2}{2\tq(\vx)}\right),\quad 
    \lc_s[f]:=-2\mhh_s Q^{\ast}\left[\m_s,\mh_sf\right],
\end{align}
and 
\begin{align}
    \ab_s=\ab(v-su_1),\quad \a_s=\lc_s^{-1}[\ab_s],\quad \bbb_s=\bbb(v-su_1),\quad \b_s=\lc_s^{-1}[\bbb_s].
\end{align}

Note that translation will not change the orthogonality, i.e. for any $s\in\r$
\begin{align}
    \int_{\r^3}\b_s\ab_s=\int_{\r^3}\bbb_s\a_s=0.
\end{align}
Taking $s$ derivative, we know
\begin{align}
    \frac{\ud}{\ud s}\int_{\r^3}\bbb_s\a_s=0,
\end{align}
which is equivalent to
\begin{align}\label{extra 07}
    \int_{\r^3}\frac{\ud\bbb_s}{\ud s}\li_s\left[\ab_s\right]+\int_{\r^3}\bbb_s\frac{\ud\li_s}{\ud s}\left[\ab_s\right]+\int_{\r^3}\bbb_s\li_s\left[\frac{\ud\ab_s}{\ud s}\right]=0.
\end{align}

For the first term in \eqref{extra 07}, due to oddness and orthogonality, we can directly verify that 
\begin{align}\label{extra 08}
    \lim_{s\rt0}\int_{\r^3}\frac{\ud\bbb_s}{\ud s}\li_s\left[\ab_s\right]=\int_{\r^3}\bbb\left(\frac{u_1\cdot v}{2T}\right)\a.
\end{align}

For the second term in \eqref{extra 07}, we have
\begin{align}
    \int_{\r^3}\bbb_s\frac{\ud\li_s}{\ud s}\left[\ab_s\right]=-\int_{\r^3}\bbb_s\li_s\frac{\ud\lc_s}{\ud s}\li_s\left[\ab_s\right]=-\int_{\r^3}\b_s\frac{\ud\lc_s}{\ud s}\left[\a_s\right].
\end{align}
Notice that for any $g(v)$
\begin{align}
   \lim_{s\rt0}\frac{\ud\lc_s}{\ud s}[g]=&-2\lim_{s\rt0}\frac{\ud}{\ud s}\left(\mhh_s Q^{\ast}\left[\m_s,\mh_sg\right]\right)\\
   =&-2\bigg\{-\frac{u_1\cdot v}{2T}\mhh Q^{\ast}\left[\m,\mh g\right]+\mhh Q^{\ast}\left[\frac{u_1\cdot v}{T}\m,\mh g\right]+\mhh Q^{\ast}\left[\m,\frac{u_1\cdot v}{2T}\mh g\right]\bigg\}\no\\
   =&-\frac{u_1\cdot v}{2T}\lc[g]-2\mhh Q^{\ast}\left[\frac{u_1\cdot v}{T}\m,\mh g\right]+\lc\left[\frac{u_1\cdot v}{2T}g\right].\no
\end{align}
Hence, we have
\begin{align}\label{extra 09}
    &\lim_{s\rt0}\int_{\r^3}\bbb_s\frac{\ud\li_s}{\ud s}\left[\ab_s\right]=-\lim_{s\rt0}\int_{\r^3}\b_s\frac{\ud\lc_s}{\ud s}\left[\a_s\right]\\
    =&-\int_{\r^3}\b\left(-\frac{u_1\cdot v}{2T}\right)\ab+2\int_{\r^3}\b\mhh Q^{\ast}\left[\frac{u_1\cdot v}{T}\m,\mh \a\right]-\int_{\r^3}\b\lc\left[\frac{u_1\cdot v}{2T}\a\right]\no\\
    =&\int_{\r^3}\b\left(\frac{u_1\cdot v}{2T}\right)\ab+2\int_{\r^3}\b\Gamma\left[\frac{u_1\cdot v}{T}\mh,\a\right]-\int_{\r^3}\bbb\left(\frac{u_1\cdot v}{2T}\right)\a.\no
\end{align}

For the third term in \eqref{extra 07}, we have
\begin{align}\label{extra 10}
    \lim_{s\rt0}\int_{\r^3}\bbb_s\li_s\left[\frac{\ud\ab_s}{\ud s}\right]=\lim_{s\rt0}\int_{\r^3}\b_s\frac{\ud\ab_s}{\ud s}=\int_{\r^3}\b\left(\frac{u_1\cdot v}{2T}\right)\ab-2\int_{\r^3}\b(u_1\cdot\bbb).
\end{align}

Inserting \eqref{extra 08}, \eqref{extra 09} and \eqref{extra 10} into \eqref{extra 07}, we have
\begin{align}
    \int_{\r^3}\bbb\left(\frac{u_1\cdot v}{2T}\right)\a+\int_{\r^3}\b\left(\frac{u_1\cdot v}{2T}\right)\ab+2\int_{\r^3}\b\Gamma\left[\frac{u_1\cdot v}{T}\mh,\a\right]-\int_{\r^3}\bbb\left(\frac{u_1\cdot v}{2T}\right)\a\\
    +\int_{\r^3}\b\left(\frac{u_1\cdot v}{2T}\right)\ab-2\int_{\r^3}\b(u_1\cdot\bbb)&=0.\no
\end{align}
Hence, we know that
\begin{align}
    \int_{\r^3}\b\Gamma\left[\frac{u_1\cdot v}{T}\mh,\a\right]=-\int_{\r^3}\b\left(\frac{u_1\cdot v}{2T}\right)\ab+\int_{\r^3}\b(u_1\cdot\bbb).
\end{align}
This verifies \eqref{extra 06}.
\end{proof}

\begin{lemma}
    We have the identity
    \begin{align}\label{extra 11}
        &\Gamma\left[\mh\bigg(\frac{\rq_1}{\rq}+\frac{\uq_1\cdot\vv}{\tq}+\frac{\tq_1(\abs{\vv}^2-3\tq)}{2\tq^2}\bigg),\mh\bigg(\frac{\rq_1}{\rq}+\frac{\uq_1\cdot\vv}{\tq}+\frac{\tq_1(\abs{\vv}^2-3\tq)}{2\tq^2}\bigg)\right]\\
        =&-\lc\left[\m\bigg(\frac{\rq_1}{\rq}+\frac{\uq_1\cdot\vv}{\tq}+\frac{\tq_1(\abs{\vv}^2-3\tq)}{2\tq^2}\bigg)^2\right].\no
    \end{align}
\end{lemma}

\begin{proof}
    The proof can be found in \cite[(60)]{Bardos.Golse.Levermore1991}. A different derivation can be achieved by considering the expansion with respect to $\e$ in $Q\left[\m_F,\m_F\right]=0$ where
    \begin{align}
        \m_F=\frac{\rho_F}{(2\pi T_F)^{\frac{3}{2}}}\exp\left(-\frac{\abs{v-u_F}^2}{2T_F}\right)=\frac{(\rho+\e\rho_1+\e^2\rho_2)}{\big(2\pi (T_0+\e T_1+\e^2T_2)\big)^{\frac{3}{2}}}\exp\left(-\frac{\abs{v-(\e u_1+\e^2u_2)}^2}{2(T_0+\e T_1+\e^2T_2)}\right).
    \end{align}
\end{proof}

\begin{lemma}
    Equation \eqref{expand 12} is equivalent to
    \begin{align}\label{extra 04}
    -\frac{P}{T}\nx\cdot\left(-\frac{2}{3}\abs{u_1}^2\id+2(u_1\otimes u_1)\right)+\nx \mathfrak{p}+\nx\cdot \left(\tau^{(1)}-\tau^{(2)}\right)=0,
    \end{align}
    where 
    \begin{align}
        \tau^{(1)}:=&\int_{\r^3}\mathscr{B}\left\{\mh\left(v\cdot\frac{\nx u_1}{T}\cdot v\right)\right\},\\
        \tau^{(2)}:=&\int_{\r^3}\mathscr{B}\left\{v\cdot\nx^2\tq\cdot\frac{\a}{2\tq^2}\right\}+\int_{\r^3}\mathscr{B}\left\{\mhh\vv\cdot\nx\left(\mh\frac{\a}{2\tq^2}\right)\cdot\nx\tq\right\}\\
        &+\int_{\r^3}\mathscr{B}\Gamma\left[\a\cdot\frac{\nx\tq}{2\tq^2},\a\cdot\frac{\nx\tq}{2\tq^2}\right].\no
    \end{align}
\end{lemma}

\begin{proof}
\eqref{expand 12} is equivalent to
\begin{align}\label{expand 17}\quad\int_{\r^3}\vv\Big(\vv\cdot\nx(\mh\f_2)\Big)\ud\vv=0.
\end{align}
Using \eqref{expand 14}, \eqref{expand 17} can be rewritten as
\begin{align}\label{expand 15}
    \nx\cdot\Bigg(-\int_{\r^3}v\otimes v\mh\li\Big[\mhh\left(\vv\cdot\nx\left(\mh\f_1\right)\right)\Big]+\int_{\r^3}v\otimes v\mh\li\Big[\Gamma[\f_1,\f_1]\Big]\\
    +\int_{\r^3}v\otimes v\m\bigg(\frac{\rq_2}{\rq}+\frac{\uq_2\cdot\vv}{\tq}+\frac{\tq_2(\abs{\vv}^2-3\tq)}{2\tq^2}\bigg)\Bigg)&=0.\no
\end{align}

\paragraph{\underline{First Term in \eqref{expand 15}}}
For the first term in \eqref{expand 15},
by orthogonality, since $\li$ is self-adjoint, using \eqref{expand 8}, we have
\begin{align}\label{expand 22}
    &-\int_{\r^3}v\otimes v\mh\li\Big[\mhh\left(\vv\cdot\nx\left(\mh\f_1\right)\right)\Big]\\
    =&-\int_{\r^3}\left(v\otimes v-\frac{\abs{v}^2}{3}\id\right)\mh\li\Big[\mhh\left(\vv\cdot\nx\left(\mh\f_1\right)\right)\Big]\no\\
    =&-\int_{\r^3}\li\left[\bbb\right]\bigg(\mhh\left(\vv\cdot\nx\left(\mh\f_1\right)\right)\bigg)\no\\
    =&-\int_{\r^3}\mathscr{B}\bigg(\mhh\left(\vv\cdot\nx\left(\mh\f_1\right)\right)\bigg),\no\\
    =&-\int_{\r^3}\mathscr{B}\left\{\mhh\vv\cdot\nx\left(-\mh\a\cdot\frac{\nx\tq}{2\tq^2}+\m\bigg(\frac{\rq_1}{\rq}+\frac{\uq_1\cdot\vv}{\tq}+\frac{\tq_1(\abs{\vv}^2-3\tq)}{2\tq^2}\bigg)\right)\right\}.\no
\end{align}
Due to oddness, the $\rq_1$ and $\tq_1$ terms in \eqref{expand 22} vanish. Hence, the first term in \eqref{expand 15} is actually
\begin{align}
    &-\int_{\r^3}\mathscr{B}\left\{\mhh\vv\cdot\nx\left(-\mh\a\cdot\frac{\nx\tq}{2\tq^2}+\m\frac{\uq_1\cdot\vv}{\tq}\right)\right\}=-\tau^{(1)}+\widetilde{\tau}^{(2)}+\widetilde{\varsigma},
\end{align}
where
\begin{align}
    \tau^{(1)}:=&\int_{\r^3}\mathscr{B}\left\{\mh\left(v\cdot\frac{\nx u_1}{T}\cdot v\right)\right\},\\
    \widetilde{\tau}^{(2)}:=&\int_{\r^3}\mathscr{B}\left\{v\cdot\nx^2\tq\cdot\frac{\a}{2\tq^2}\right\}+\int_{\r^3}\mathscr{B}\left\{\mhh\vv\cdot\nx\left(\mh\frac{\a}{2\tq^2}\right)\cdot\nx\tq\right\},
\end{align}
and
\begin{align}
    \widetilde{\varsigma}:=&\int_{\r^3}\mathscr{B}\left\{\mh\left(v\cdot\frac{\nx\tq}{T^2}\right)(u_1\cdot v)\right\}-\int_{\r^3}\mathscr{B}\left\{\mhh\left(v\cdot\nx\m\right)\left(\frac{u_1\cdot v}{T}\right)\right\}\\
    =&\int_{\r^3}\mathscr{B}\left\{\mh\left(v\cdot\frac{\nx\tq}{T^2}\right)(u_1\cdot v)\right\}-\int_{\r^3}\mathscr{B}\left\{\mh\left(v\cdot\frac{\nx\tq}{2\tq^3}\right)\left(\abs{v}^2-5\tq\right)\left(u_1\cdot v\right)\right\}\no\\
    =&\frac{\nx\tq}{T^2}\cdot\int_{\r^3}\mathscr{B}\cdot\left\{(u_1\cdot v)v\mh\right\}-\frac{\nx\tq}{2\tq^3}\cdot\int_{\r^3}\mathscr{B}\cdot\left\{\ab\left(u_1\cdot v\right)\right\}.\no
\end{align}

\paragraph{\underline{Second Term of \eqref{expand 15}}}
For the second term of \eqref{expand 15}, we have
\begin{align}\label{expand 23}
    &\int_{\r^3}v\otimes v\mh\li\Big[\Gamma[\f_1,\f_1]\Big]
    =\int_{\r^3}\bar{\mathscr{B}}\li\Big[\Gamma[\f_1,\f_1]\Big]=\int_{\r^3}\mathscr{B}\Gamma[\f_1,\f_1]\\
    =&\int_{\r^3}\mathscr{B}\Gamma\left[-\a\cdot\frac{\nx\tq}{2\tq^2},-\a\cdot\frac{\nx\tq}{2\tq^2}\right]\no\\
    &+2\int_{\r^3}\mathscr{B}\Gamma\left[-\a\cdot\frac{\nx\tq}{2\tq^2},\mh\bigg(\frac{\rq_1}{\rq}+\frac{\uq_1\cdot\vv}{\tq}+\frac{\tq_1\abs{\vv}^2-3\tq)}{2\tq^2}\bigg)\right]\no\\
    &+\int_{\r^3}\mathscr{B}\Gamma\left[\mh\bigg(\frac{\rq_1}{\rq}+\frac{\uq_1\cdot\vv}{\tq}+\frac{\tq_1(\abs{\vv}^2-3\tq)}{2\tq^2}\bigg),\mh\bigg(\frac{\rq_1}{\rq}+\frac{\uq_1\cdot\vv}{\tq}+\frac{\tq_1(\abs{\vv}^2-3\tq)}{2\tq^2}\bigg)\right].\no
\end{align}
For the first term in \eqref{expand 23}, denote
\begin{align}
    &\overline{\tau}^{(2)}:=\int_{\r^3}\mathscr{B}\Gamma\left[-\a\cdot\frac{\nx\tq}{2\tq^2},-\a\cdot\frac{\nx\tq}{2\tq^2}\right].
\end{align}
Then denote 
\begin{align}
    \tau^{(2)}:=\widetilde{\tau}^{(2)}+\overline{\tau}^{(2)}.
\end{align}
For the second term in \eqref{expand 23}, using identity \eqref{extra 06}, we obtain
\begin{align}
    \overline{\varsigma}:=&2\int_{\r^3}\mathscr{B}\Gamma\left[-\a\cdot\frac{\nx\tq}{2\tq^2},\mh\bigg(\frac{\rq_1}{\rq}+\frac{\uq_1\cdot\vv}{\tq}+\frac{\tq_1(\abs{\vv}^2-3\tq)}{2\tq^2}\bigg)\right]\\
    =&-\int_{\r^3}\mathscr{B}\Gamma\left[\a\cdot\frac{\nx\tq}{\tq^2},\mh\bigg(\frac{\uq_1\cdot\vv}{\tq}\bigg)\right]=-\frac{\nx\tq}{\tq^3}\cdot\int_{\r^3}\b\Gamma\left[\a,\mh\left(\uq_1\cdot\vv\right)\right]\no\\
    =&\frac{\nx\tq}{2\tq^3}\cdot\int_{\r^3}\b\cdot\ab(u_1\cdot v)-\frac{\nx\tq}{\tq^2}\cdot\int_{\r^3}\b\cdot (u_1\cdot\bbb).\no
\end{align}
Then we have
\begin{align}
    \widetilde{\varsigma}+\overline{\varsigma}=0.
\end{align}

For the third term in \eqref{expand 23}, direct computation using \eqref{extra 11} and oddness reveals that
\begin{align}
    &-\int_{\r^3}\mathscr{B}\Gamma\left[\mh\bigg(\frac{\rq_1}{\rq}+\frac{\uq_1\cdot\vv}{\tq}+\frac{\tq_1(\abs{\vv}^2-3\tq)}{2\tq^2}\bigg),\mh\bigg(\frac{\rq_1}{\rq}+\frac{\uq_1\cdot\vv}{\tq}+\frac{\tq_1(\abs{\vv}^2-3\tq)}{2\tq^2}\bigg)\right]\\
    =&\int_{\r^3}\bbb\lc^{-1}\left[\lc\left[\mh\frac{(u_1\cdot v)^2}{T^2}\right]\right]=-\frac{2P}{3T}\abs{u_1}^2\id+\frac{2P}{T}(u_1\otimes u_1).\no
\end{align}

\paragraph{\underline{Third Term of \eqref{expand 15}}}
For the third term of \eqref{expand 15}, due to oddness, $u_2$ terms vanish, and thus we have
\begin{align}
    \int_{\r^3}v\otimes v\m\bigg(\frac{\rq_2}{\rq}+\frac{\uq_2\cdot\vv}{\tq}+\frac{\tq_2(\abs{\vv}^2-3\tq)}{2\tq^2}\bigg)=\big(\tq\rq_2+\rq\tq_2\big)\id.
\end{align}
\end{proof}

\subsubsection{\underline{Ghost-Effect Equations}}

Collecting all above and rearranging the terms, we have
\begin{align}
\m(\vx,\vv)=\frac{\rq(\vx)}{\big(2\pi\tq(\vx)\big)^{\frac{3}{2}}}
\exp\left(-\frac{\abs{\vv}^2}{2\tq(\vx)}\right),
\end{align}
and
\begin{align}\label{extra 34}
    \f_1=&-\a\cdot\frac{\nx\tq}{2\tq^2}+\mh\bigg(\frac{\rq_1}{\rq}+\frac{\uq_1\cdot\vv}{\tq}+\frac{\tq_1\big(\abs{\vv}^2-3\tq\big)}{2\tq^2}\bigg),\\
    \f_2=&-\li\Big[\mhh\left(\vv\cdot\nx\left(\mh\f_1\right)\right)\Big]+\li\Big[\Gamma[\f_1,\f_1]\Big]
    +\mh\bigg(\frac{\rq_2}{\rq}+\frac{\uq_2\cdot\vv}{\tq}+\frac{\tq_2\big(\abs{\vv}^2-3\tq\big)}{2\tq^2}\bigg),\label{extra 35}
\end{align}
where $(\rho,0,T)$, $(\rho_1,u_1,T_1)$ and $(\rho_2,u_2,T_2)$ satisfy
\begin{itemize}
    \item 
    Order $1$ equation:
    \begin{align}
        \nx P=\nx(\rq\tq)=&0\label{pressure}.
    \end{align}
    \item 
    Order $\e$ system:
    \begin{align}
    \nx\cdot(\rq\uq_1)=&0,\label{continuity}\\
    \nx P_1=\nx\big(T\rho_1+\rho T_1\big)=&0,\\
    \nx\cdot\left(\k\dfrac{\nx\tq}{2\tq^2}\right)=&5P\big(\nx\cdot\uq_1\big)\label{temp}.
    \end{align}
    \item 
    Order $\e^2$ system:
    \begin{align}
    \rq(\uq_1\cdot\nx\uq_1)+\nx \mathfrak{p}=&\nx\cdot\left(\tau^{(1)}-\tau^{(2)}\right)\label{velocity}.
    \end{align}
\end{itemize}
Here $u_k=(u_{k,1},u_{k,2},u_{k,3})$,
\begin{align}\label{aa 22}
    P:=\rho T,\quad P_1:=T\rho_1+\rho T_1,\quad \mathfrak{p}:=T\rho_2+\rho T_2,
\end{align}
\begin{align}
    \ab:=v\cdot\left(\abs{v}^2-5T\right)\mh,\quad \a:=\li[\ab]=\li\left[v\cdot\left(\abs{v}^2-5T\right)\mh\right],\quad \k\id:=\int_{\r^3}\a\otimes\ab\ud v,
\end{align}
and
\begin{align}
    \tau^{(1)}:=&\lambda\left(\nx\uq_1+\big(\nx\uq_1\big)^t-\frac{2}{3}\big(\nx\cdot\uq_1\big)\id\right),\\ \tau^{(2)}:=&\frac{\lambda^2}{P}\bigg(K_1\Big(\nx^2\tq-\frac{1}{3}\dx\tq\id\Big)+\frac{K_2}{\tq}\Big(\nx\tq\otimes\nx\tq-\frac{1}{3}\abs{\nx\tq}^2\id\Big)\bigg)
\end{align}
for smooth functions $\lambda[T]>0$, and positive constants $K_1$ and $K_2$ \cite{Sone2002, Bobylev1995, Kogan.Galkin.Fridlender1976}.

We observe that \eqref{pressure}, \eqref{continuity}, \eqref{temp} and \eqref{velocity}  are a set of equations sufficient to determine $(\rq,\uq_1,\tq, \nx\mathfrak{p})$ uniquely once suitable boundary conditions are specified:
{\begin{align}\label{fluid system=}
\left\{
\begin{array}{rcl}
    \nx P=\nx(\rq\tq)&=&0,\\\rule{0ex}{1.5em}
    \rq(\uq_1\cdot\nx\uq_1)+\nx \mathfrak{p}&=&\nx\cdot\left(\tau^{(1)}-\tau^{(2)}\right),\\\rule{0ex}{1.5em}
    \nx\cdot(\rq\uq_1)&=&0,\\\rule{0ex}{2em}
    \nx\cdot\left(\k\dfrac{\nx\tq}{2\tq^2}\right)&=&5P\big(\nx\cdot\uq_1\big),
    \end{array}
    \right.
\end{align}}
Notice that $\mathfrak{p}$ enters in the equations only through its gradient so we are free to choose a definite value by imposing $\ds\int_{\Omega}\mathfrak{p}=0$.

Also, we are left with an additional requirement:
\begin{align}\label{fluid system'}
    \nx P_1=\nx\big(T\rho_1+\rho T_1\big)=&0.
\end{align}
The higher-order terms of the expansion will be discussed in Section 3.

\subsection{Normal Chart near Boundary}

In order to define the boundary layer correction, we need to design a coordinate system based on the normal and tangential directions on the boundary surface. 
Our main goal is to rewrite the three-dimensional transport operator $\vv\cdot\nx$ in this new coordinate system. This is basically textbook-level differential geometry, so we omit the details.

\paragraph{Substitution 1: Spatial Substitution:}
For a smooth manifold $\p\Omega$, there exists an orthogonal curvilinear coordinates system $(\iota_1,\iota_2)$ such that the coordinate lines coincide with the principal directions at any $\vx_0\in\p\Omega$ (at least locally).

Assume $\p\Omega$ is parameterized by $\vr=\vr(\iota_1,\iota_2)$. Let $\abs{\cdot}$ denote the length. Hence, $\p_{\iota_1}\vr$ and $\p_{\iota_2}\vr$ represent two orthogonal tangential vectors. Denote $\pl_i=\abs{\p_{\iota_i}\vr}$ for $i=1,2$. Then define the two orthogonal unit tangential vectors
\begin{align}\label{coordinate 14}
\vt_1:=\frac{\p_{\iota_1}\vr}{\pl_1},\ \ \vt_2:=\frac{\p_{\iota_2}\vr}{\pl_2}.
\end{align}
Also, the outward unit normal vector is
\begin{align}\label{coordinate 1}
\vn:=\frac{\p_{\iota_1}\vr\times\p_{\iota_2}\vr}{\abs{\p_{\iota_1}\vr\times\p_{\iota_2}\vr}}=\vt_1\times\vt_2.
\end{align}
Obviously, $(\vt_1,\vt_2,\vn)$ forms a new orthogonal frame. Hence, consider the corresponding new coordinate system $(\iota_1,\iota_2,\mn)$, where $\mn$ denotes the normal distance to boundary surface $\p\Omega$, i.e.
\begin{align}
\vx=\vr-\mn\vn.
\end{align}
Note that $\mn=0$ means $\vx\in\p\Omega$ and $\mn>0$ means $\vx\in\Omega$ (before reaching the other side of $\p\Omega$). Using this new coordinate system and denoting $\k_i$ the principal curvatures, the transport operator becomes
\begin{align}\label{coordinate 11}
\vv\cdot\nx=-(\vv\cdot\vn)\frac{\p}{\p\mn}-\frac{\vv\cdot\vt_1}{\pl_1(\kk_1\mn-1)}\frac{\p}{\p\iota_1}-\frac{\vv\cdot\vt_2}{\pl_2(\kk_2\mn-1)}\frac{\p}{\p\iota_2}.
\end{align}

\paragraph{Substitution 2: Velocity Substitution:}
Define the orthogonal velocity substitution for $\vvv:=(\va,\vb,\vc)$ as
\begin{align}\label{aa 44}
\left\{
\begin{array}{l}
-\vv\cdot\vn:=\va,\\
-\vv\cdot\vt_1:=\vb,\\
-\vv\cdot\vt_2:=\vc.
\end{array}
\right.
\end{align}
Then the transport operator becomes
\begin{align}
\vv\cdot\nx=&\va\frac{\p}{\p\mn}-\frac{1}{R_1-\mn}\bigg(\vb^2\frac{\p}{\p\va}-\va\vb\frac{\p}{\p\vb}\bigg)
-\frac{1}{R_2-\mn}\bigg(\vc^2\frac{\p}{\p\va}-\va\vc\frac{\p}{\p\vc}\bigg)\\
&+\frac{1}{\pl_1\pl_2}\Bigg(\frac{R_1\p_{\iota_1\iota_1}\vr\cdot\p_{\iota_2}\vr}{\pl_1(R_1-\mn)}\vb\vc
+\frac{R_2\p_{\iota_1\iota_2}\vr\cdot\p_{\iota_2}\vr}{\pl_2(R_2-\mn)}\vc^2\Bigg)\frac{\p}{\p\vb}\no\\
&+\frac{1}{\pl_1\pl_2}\Bigg(R_2\frac{\p_{\iota_2\iota_2}\vr\cdot\p_{\iota_1}\vr}{\pl_2(R_2-\mn)}\vb\vc
+\frac{R_1\p_{\iota_1\iota_2}\vr\cdot\p_{\iota_1}\vr}{\pl_1(R_1-\mn)}\vb^2\Bigg)\frac{\p}{\p\vc}\no\\
&+\bigg(\frac{R_1\vb}{\pl_1(R_1-\mn)}\frac{\p}{\p\iota_1}+\frac{R_2\vc}{\pl_2(R_2-\mn)}\frac{\p}{\p\iota_2}\bigg)\no,
\end{align}
where $R_i=\k_i^{-1}$ represent the radii of principal curvature. 

\paragraph{Substitution 3: Scaling Substitution:}
Finally, we define the scaled variable $\eta=\dfrac{\mn}{\e}$, which implies $\dfrac{\p}{\p\mn}=\dfrac{1}{\e}\dfrac{\p}{\p\eta}$. Then the transport operator becomes
\begin{align}
\vv\cdot\nx=&\frac{1}{\e}\va\dfrac{\p }{\p\eta}-\dfrac{1}{R_1-\e\eta}\bigg(\vb^2\dfrac{\p }{\p\va}-\va\vb\dfrac{\p}{\p\vb}\bigg)
-\dfrac{1}{R_2-\e\eta}\bigg(\vc^2\dfrac{\p }{\p\va}-\va\vc\dfrac{\p }{\p\vc}\bigg)\\
&+\frac{1}{\pl_1\pl_2}\Bigg(\frac{R_1\p_{\iota_1\iota_1}\vr\cdot\p_{\iota_2}\vr}{\pl_1(R_1-\e\eta)}\vb\vc
+\frac{R_2\p_{\iota_1\iota_2}\vr\cdot\p_{\iota_2}\vr}{\pl_2(R_2-\e\eta)}\vc^2\Bigg)\frac{\p}{\p\vb}\no\\
&+\frac{1}{\pl_1\pl_2}\Bigg(R_2\frac{\p_{\iota_2\iota_2}\vr\cdot\p_{\iota_1}\vr}{\pl_2(R_2-\e\eta)}\vb\vc
+\frac{R_1\p_{\iota_1\iota_2}\vr\cdot\p_{\iota_1}\vr}{\pl_1(R_1-\e\eta)}\vb^2\Bigg)\frac{\p}{\p\vc}\no\\
&+\bigg(\frac{R_1\vb}{\pl_1(R_1-\e\eta)}\frac{\p}{\p\iota_1}+\frac{R_2\vc}{\pl_2(R_2-\e\eta)}\frac{\p}{\p\iota_2}\bigg).\no
\end{align}

\subsection{Milne Problem with Tangential Dependence}

To construct the Hilbert expansion in a general domain, it is important to
study the Milne problem depending on the tangential variable $(\iota_1,\iota_2)$. Notice that, in the new variables, $\mb=\mb(\iota_1,\iota_2,\vvv)$. Set
\begin{align}
\lc_w[f]:=-2\m_w^{-\frac{1}{2}} Q^{\ast}\left[\m_w,\m_w^{\frac{1}{2}} f\right]=\nu_w f-K_w[f].
\end{align}
Let
$\blf(\eta,\iota_1,\iota_2,\vvv)$ be solution to the Milne problem
\begin{align}\label{Milne}
    \va\dfrac{\p\blf }{\p\eta}+\nu_w \blf-K_w\left[\blf \right]=0,
\end{align}
with in-flow boundary condition at $\eta=0$
\begin{align}\label{aa 06}
    \blf (0,\iota_1,\iota_2,\vvv)=\a\cdot\frac{\nx T}{2T^2}\ \ \text{for}\ \ \va>0,
\end{align} 
and the zero mass-flux condition
\begin{align}\label{aa 05}
    \int_{\r^3}\va\mbh(\iota_1,\iota_2,\vvv)\blf (0,\iota_1,\iota_2,\vvv)\ud \vvv=0.
\end{align}

\begin{theorem}\label{boundary well-posedness}
Assume that $\nx\tq\in W^{k,\infty}(\p\Omega)$ for some $k\in\mathbb{N}$ and $\lnms{\tq}{\p\Omega}\ls 1$. Then there exists 
\begin{align}\label{extra 17}
    \blf_{\infty}(\iota_1,\iota_2,\vvv):=\blf_{\infty}(\iota_1,\iota_2,\vv):=\mbh\left(\frac{\rq^B}{\rq_w}+\frac{\uq^B\cdot\vv}{\tb}+\frac{\tq^B(\abs{\vv}^2-3\tb)}{2\tb^2}\right)\in\nk,
\end{align}
for $\rq_w:=P\tb^{-1}$ and some $\Big(\rq^B(\iota_1,\iota_2),\uq^B(\iota_1,\iota_2),\tq^B(\iota_1,\iota_2)\Big)$ such that 
\begin{align}
    \int_{\r^3}\va\mbh(\vvv)\blf_{\infty}(\vvv)\ud \vvv=0,
\end{align}
and a unique solution $\blf(\eta,\iota_1,\iota_2,\vvv)$ to \eqref{Milne} such that $\blff:=\blf-\blf_{\infty}$ satisfies 
\begin{align}\label{Milne.}
        \left\{
        \begin{array}{l}
        \va\dfrac{\p\blff }{\p\eta}+\nu_w \blff-K_w\left[\blff\right]=0,\\\rule{0ex}{1.2em}
        \blff(0,\iota_1,\iota_2,\vvv)=\blf(0,\iota_1,\iota_2,\vvv)-   \blf_{\infty}(\iota_1,\iota_2,\vvv)\ \ \text{for}\ \ \va>0,\\\rule{0ex}{2.0em}
        \ds\int_{\r^3}\va\mbh(\iota_1,\iota_2,\vvv)\blff(0,\iota_1,\iota_2,\vvv)\ud \vvv=0.
        \end{array}
        \right.
\end{align}
and for some $K_0>0$ and any $0<\N\leq k$
\begin{align}
    \abs{\blf_{\infty}}+\lnmm{\ue^{K_0\eta}\blff}\ls& \lnms{\nx\tq}{\p\Omega},\label{final 41}\\ \lnmm{\ue^{K_0\eta}\va\p_{\eta}\blff}+\lnmm{\ue^{K_0\eta}\va\p_{\va}\blff}\ls& \lnms{\nx\tq}{\p\Omega},\label{final 42}\\
    \lnmm{\ue^{K_0\eta}\p_{\vb}\blff}+\lnmm{\ue^{K_0\eta}\p_{\vc}\blff}\ls& \lnms{\nx\tq}{\p\Omega},\label{final 43}\\
    \lnmm{\ue^{K_0\eta}\p_{\iota_1}^{\N}\blff}+\lnmm{\ue^{K_0\eta}\p_{\iota_2}^{\N}\blff}\ls& \lnms{\nx\tq}{\p\Omega}+\sum_{j=1}^{\N}\lnms{\p_{\iota_1}^j\nx\tq}{\p\Omega}+\sum_{j=1}^{\N}\lnms{\p_{\iota_2}^j\nx\tq}{\p\Omega}.\label{final 44}
\end{align}
\end{theorem}

\begin{proof}
    Based on \cite{Bardos.Caflisch.Nicolaenko1986} and \cite{BB002}, we have the well-posedness of \eqref{Milne}. Also, estimates \eqref{final 41}\eqref{final 42} and \eqref{final 43} follow. Hence, we will focus on \eqref{final 44}. Let $W:=\dfrac{\p\blff}{\p\iota_i}$ for $i=1,2$. Then $W$ satisfies
    \begin{align}\label{Milne tangential}
        \left\{
        \begin{array}{l}
        \va\dfrac{\p W }{\p\eta}+\nu_w W-K_w\left[W\right]=-\dfrac{\p\nu_w}{\p\iota_i}\blff+\dfrac{\p K_w}{\p\iota_i}\left[\blff\right],\\\rule{0ex}{1.2em}
        W(0,\iota_1,\iota_2,\vvv)=-\dfrac{\p}{\p\iota_i}\left(\a\cdot\dfrac{\nx T}{2T^2}\right)-\dfrac{\p\blff_{\infty}}{\p\iota_i}(\iota_1,\iota_2,\vvv)\ \ \text{for}\ \ \va>0,\\\rule{0ex}{1.5em}
        \ds\int_{\r^3}\va\mbh(\vvv)W(0,\iota_1,\iota_2,\vvv)\ud \vvv=-\ds\int_{\r^3}\va\frac{\p\mbh}{\p\iota_i}(\iota_1,\iota_2,\vvv)\blff(0,\iota_1,\iota_2,\vvv)\ud \vvv.
        \end{array}
        \right.
    \end{align}
    Multiplying $\abs{\vvv}^2\mbh$ on both sides of \eqref{Milne.} and integrating over $\r^3$ yield
    \begin{align}
        \int_{\r^3}\va\abs{\vvv}^2\mbh(\iota_1,\iota_2,\vvv)\blff(0,\iota_1,\iota_2,\vvv)\ud\vvv=\int_{\r^3}\va\abs{\vvv}^2\mbh(\iota_1,\iota_2,\vvv)\blff(\infty,\iota_1,\iota_2,\vvv)\ud\vvv=0,
    \end{align}
    which, combined with the zero mass-flux of $\blff$, further implies
    \begin{align}
        \int_{\r^3}\va\frac{\p\mbh}{\p\iota_i}(\iota_1,\iota_2,\vvv)\blff(0,\iota_1,\iota_2,\vvv)\ud \vvv=0.
    \end{align}
    Hence, $W$ still satisfies the zero mass-flux condition. Also, notice that 
    \begin{align}
        \lnmm{\ue^{K_0\eta}\left(-\dfrac{\p\nu_w}{\p\iota_i}\blff+\dfrac{\p K_w}{\p\iota_i}\left[\blff\right]\right)}\ls& \lnms{\nx\tq}{\p\Omega}.
    \end{align}
    Therefore, based on \cite{Bardos.Caflisch.Nicolaenko1986}, there exists a unique $W_{\infty}\in\nk$ such that
    \begin{align}
        \abs{W_{\infty}}+\lnmm{\ue^{K_0\eta}\big(W-W_{\infty}\big)}\ls& \lnms{\nx\tq}{\p\Omega}+\lnms{\p_{\iota_i} \nx\tq}{\p\Omega}.
    \end{align}
    In particular, since $\blff\rt0$ as $\eta\rt\infty$, we must have $W_{\infty}=0$. Hence, \eqref{final 44} is verified for $\N=1$. The $\N>1$ cases follow inductively.
\end{proof}

Let $\chi(y)\in C^{\infty}(\r)$
and $\ch(y)=1-\chi(y)$ be smooth cut-off functions satisfying
\begin{align}
    \chi(y)=\left\{
    \begin{array}{ll}
    1&\ \ \text{if}\ \ \abs{y}\leq1,\\
    0&\ \ \text{if}\ \ \abs{y}\geq2,
    \end{array}
    \right.
\end{align}
In view of the later regularity estimates (see the companion paper \cite{AA023}), we define a cutoff boundary layer $\fb_1$:
\begin{align}\label{extra 15}
    \fb_1(\eta,\iota_1,\iota_2,\vvv):=\ch\left(\e^{-1}\va\right)\chi(\e\eta)\blff (\eta,\iota_1,\iota_2,\vvv).
\end{align}
We can verify that $\fb_1$ satisfies
\begin{align}
    \va\dfrac{\p\fb_1 }{\p\eta}+\nu_w \fb_1-K_w\left[\fb_1\right]=&\va\ch(\e^{-1}\va)\frac{\p\chi(\e\eta)}{\p\eta}\blff+\chi(\e\eta)\bigg(\ch(\e^{-1}\va)K_w\Big[\blff\Big]-K_w\Big[\ch(\e^{-1}\va)\blff\Big]\bigg),
\end{align}
with
\begin{align}\label{extra 22}
    \fb_1(0,\iota_1,\iota_2,\vvv)=\ch\left(\e^{-1}\va\right)\Big(-\a\cdot\frac{\nx T}{2T^2}-\blf_{\infty}(\iota_1,\iota_2,\vvv)\Big)\ \ \text{for}\ \ \va>0.
\end{align}
Due to the cutoff $\ch$, $\fb_1$ cannot preserve the zero mass-flux condition, i.e.
\begin{align}\label{aa 37}
    \int_{\r^3}\va\mbh(\iota_1,\iota_2,\vvv)\fb_1(0,\iota_1,\iota_2,\vvv)\ud \vvv=&\int_{\r^3}\va\mbh(\iota_1,\iota_2,\vvv)\ch\left(\e^{-1}\va\right)\blff(0,\iota_1,\iota_2,\vvv)\ud \vvv\\
    =&\int_{\r^3}\va\mbh(\iota_1,\iota_2,\vvv)\chi\left(\e^{-1}\va\right)\blff(0,\iota_1,\iota_2,\vvv)\ud \vvv\ls\oot\e.\no
\end{align}
The zero mass-flux condition will be restored with the help of $\f_2$ in \eqref{aa 38}.

\subsection{Analysis of Boundary Matching}\label{sec:matching}

Considering the boundary condition in \eqref{large system} and the expansion \eqref{aa 08}, we require the matching condition for $\vx_0\in\p\Omega$ and $v\cdot\vn<0$:
\begin{align}
    \mb\Big|_{v\cdot n<0}&=\ms\displaystyle\int_{\vv'\cdot\vn>0}
    \mb\abs{\vv'\cdot\vn}\ud{\vv'},\label{aa 02}\\
    \mbh\left(\f_1+\fb_1\right)\Big|_{v\cdot n<0}&=\ms\displaystyle\int_{\vv'\cdot\vn>0}
    \mbh\left(\f_1+\fb_1\right)\abs{\vv'\cdot\vn}\ud{\vv'}+O(\e).\label{aa 03}
\end{align}

In order to guarantee \eqref{aa 02}, we deduce that
\begin{align}\label{aa 25}
    T(x_0)=\tb(x_0).
\end{align}
This determines the boundary conditions for $T$.

In order to guarantee \eqref{aa 03}, due to \eqref{aa 37}, it suffices to require that at $\eta=0$
\begin{align}\label{aa 10}
    \mbh\big(\f_1+\blf -\blf_{\infty}\big)\Big|_{v\cdot n<0}&=\ms\displaystyle\int_{\vv'\cdot\vn>0}
    \mbh\big(\f_1+\blf -\blf_{\infty}\big)\abs{\vv'\cdot\vn}\ud{\vv'}.
\end{align}

\begin{lemma}
    With the boundary condition \eqref{aa 06} for \eqref{Milne}, and for $x_0\in\p\Omega$
    \begin{align}\label{aa 24}
        u_1(x_0)=u^B,\quad T_1(x_0)=T^B,
    \end{align}
    \eqref{aa 10} is valid.
\end{lemma}
\begin{proof}
Using \eqref{extra 34} and \eqref{extra 17}, we have for $\vx_0\in\p\Omega$
\begin{align}
    \f_1+\blf -\blf_{\infty}=&\a\cdot\frac{\nx\tq}{2\tq^2}+\mh\bigg(\frac{\rq_1}{\rq}+\frac{\uq_1\cdot\vv}{\tq}+\frac{\tq_1\big(\abs{\vv}^2-3\tq\big)}{2\tq^2}\bigg)\\
    &+\blf-\mh\left(\frac{\rq^B}{\rq}+\frac{\uq^B\cdot\vv}{\tq}+\frac{\tq^B(\abs{\vv}^2-3\tb)}{2\tq^2}\right).\no
\end{align}
With \eqref{aa 24}, we have
\begin{align}
    \f_1+\blf -\blf_{\infty}=&\left(\blf+\a\cdot\frac{\nx\tq}{2\tq^2}\right)+\mh\bigg(\frac{\rq_1}{\rq}-\frac{\rq^B}{\rq}\bigg).
\end{align}
Since direct computation reveals that
\begin{align}
    \mh\bigg(\frac{\rq_1}{\rq}-\frac{\rq^B}{\rq}\bigg)\bigg|_{v\cdot n<0}&=\ms\displaystyle\int_{\vv'\cdot\vn>0}
    \mh\bigg(\frac{\rq_1}{\rq}-\frac{\rq^B}{\rq}\bigg)\abs{\vv'\cdot\vn}\ud{\vv'},
\end{align}
in order to verify \eqref{aa 10}, it suffices to require
\begin{align}\label{extra 36}
    \left(\blf+\a\cdot\frac{\nx\tq}{2\tq^2}\right)\bigg|_{v\cdot n<0}&=\ms\displaystyle\int_{\vv'\cdot\vn>0}
    \left(\blf+\a\cdot\frac{\nx\tq}{2\tq^2}\right)\abs{\vv'\cdot\vn}\ud{\vv'}.
\end{align}
When \eqref{aa 06} is valid, we know that
\begin{align}\label{aa 20}
    \left(\blf +\a\cdot\frac{\nx T}{2T^2}\right)\bigg|_{v\cdot n<0}=0,
\end{align}
Also, due to \eqref{aa 05} and orthogonality of $\a$, we have
\begin{align}\label{extra 23}
    \ms\displaystyle\int_{\r^3}
    \mh\blf (\vv'\cdot\vn)\ud{\vv'}=\ms\displaystyle\int_{\r^3}
    \mh\left(\a\cdot\frac{\nx T}{2T^2}\right)(\vv'\cdot\vn)\ud{\vv'}=0,
\end{align}
which, combined with \eqref{aa 20}, yields
\begin{align}
    \ms\displaystyle\int_{v'\cdot n>0}
    \mh\left(\blf +\a\cdot\frac{\nx T}{2T^2}\right)\abs{\vv'\cdot\vn}\ud{\vv'}=0.
\end{align}
Then clearly \eqref{extra 36} is true.   
\end{proof}

\section{Construction of Expansion}\label{sec:construction}

In this section, we will present the detailed construction of ghost-effect solution, $\f_1$, $\f_2$ and $\fb_1$ based on the analysis in Section \ref{sec:matching}. Since the boundary conditions are tangled together, we divide the construction into several stages.

\subsection{Construction of Boundary Layer $\fb_1$ - Stage I}

Since \eqref{aa 06} involves $\nabla_{n}T$, which is not fully provided by $\tb$, we will have to split the tangential and normal parts of the boundary layer 
\begin{align}
    \fb_1=\fb_{1,\iota_1}+\fb_{1,\iota_2}+\fb_{1,n},\quad \blf=\blf_{\iota_1}+\blf_{\iota_2}+\blf_{n}.
\end{align}
Define 
\begin{align}
    \blf_{\iota_i}:=\big(\p_{\iota_i}\tb\big)\mathcal{H}^{(i)},
\end{align}
where $\mathcal{H}^{(i)}$ for $i=1,2$ solves the Milne problem
\begin{align}\label{extra 12}
\left\{
\begin{array}{l}
    \va\dfrac{\p\mathcal{H}^{(i)}}{\p\eta}+\lc_w\left[\mathcal{H}^{(i)}\right]=0,\\
    \mathcal{H}^{(i)}(0,\vvv)=-\dfrac{\a\cdot\vt_i}{2T^2}\ \ \text{for}\ \ \va>0,\\\rule{0ex}{1.5em}
    \lim_{\eta\rt\infty}\mathcal{H}^{(i)}(\eta,\vvv)=\mathcal{H}^{(i)}_{\infty}\in\nk,
    \end{array}
    \right.
\end{align}
with the zero mass-flux condition
\begin{align}
    \int_{\r^3}\va\mbh(\vvv)\mathcal{H}^{(i)}=0.
\end{align}
Denote
\begin{align}
    \blf_{\iota_i,\infty}:=\big(\p_{\iota_i}\tb\big)\mathcal{H}^{(i)}_{\infty}.
\end{align}

Since we lack the information of of $\blf_n$ at this stage, we are not able to determine the boundary condition $T_1=T^B$ yet. However, we can fully determine the boundary condition $u_1=u^B$. Denote $u_1=(u_{1,\iota_1},u_{1,\iota_2},u_{1,n})$ for the two tangential components $(u_{1,\iota_1},u_{1,\iota_2})$ and one normal component $u_{1,n}$. Due to \eqref{aa 05}, we have
\begin{align}\label{aa 26}
    u_{1,n}(x_0)=0.
\end{align}
Due to oddness, the projection of $\mathcal{H}^{(i)}$ and $\mathcal{H}^{(i)}_{\infty}$ on $v\mbh$ only has contribution on $(v\cdot\vt_i)\mh$.
Hence, from \eqref{aa 24}, we deduce
\begin{align}
    u_{1,\iota_1}(x_0)=\beta_{1}[\tb]\p_{\iota_1}T_w,\quad u_{1,\iota_2}(x_0)=\beta_{2}[\tb]\p_{\iota_2}T_w,
\end{align}
where $\beta_{i}$ are functions depending on $\tb$. Due to isotropy, we know that $\beta_1=\beta_2$, and we denote it $\beta_0$. Hence, we arrive at 
\begin{align}\label{aa 27}
    u_{1,\iota_1}(x_0)=\beta_0[\tb]\p_{\iota_1}T_w,\quad u_{1,\iota_2}(x_0)=\beta_0[\tb]\p_{\iota_2}T_w.
\end{align}

\begin{lemma}\label{lem:u-boundary}
    Under the assumption \eqref{assumption:boundary}, for any $\NN\geq1$, the boundary data of $\uq_1$ satisfies
    \begin{align}
    \abs{\uq_1}_{W^{3,\infty}}\ls\oot.
    \end{align}
\end{lemma}

\begin{proof}
Taking $\iota_i$ derivatives for $i=1,2$ on both sides of \eqref{extra 12}, Using \eqref{assumption:boundary} and \eqref{final 41}, we conclude that
\begin{align}
    \abs{u_{1,\iota_1}}_{W^{3,\infty}}+\abs{u_{1,\iota_2}}_{W^{3,\infty}}\ls\oot
\end{align}
Then using \eqref{aa 26}, we obtain the desired estimates. 
\end{proof}

\begin{remark}
Note that the boundary condition of $\uq_1$ only depends on $\tb$ and $\nabla\tb$ directly without referring to $\tq$ in the bulk.
\end{remark}

\subsection{Well-Posedness of Ghost-Effect Equation}

Based on our analysis above, the ghost-effect equation \eqref{fluid system-} will be accompanied with the boundary conditions \eqref{aa 25}, \eqref{aa 26} and \eqref{aa 27}.
\begin{align}\label{extra 13}
    &\tq(x_0)=\tb,\quad u_{1,\iota_1}(x_0)=\beta_0[\tb]\p_{\iota_1}T_w,\quad u_{1,\iota_2}(x_0)=\beta_0[\tb]\p_{\iota_2}T_w,\quad u_{1,n}(x_0)=0.
\end{align}

\begin{theorem}\label{thm:ghost}
    Under the assumption \eqref{assumption:boundary}, for any given $P>0$, there exists a unique solution $(\rq,\uq_1,\tq; \mathfrak{p})$ (   $\mathfrak{p}$ has zero average) to the ghost-effect equation \eqref{fluid system-} with the boundary condition \eqref{extra 13} satisfying for any $\NN\in[2,\infty)$
    \begin{align}
    \nm{u_1}_{W^{3,\NN}}+\nm{\mathfrak{p}}_{W^{2,\NN}}+\nm{T-1}_{W^{4,\NN}}\ls\oot.
    \end{align}
\end{theorem}

\begin{proof}
\ \\
\paragraph{\underline{Simplified Equations}}
Denote $\bu:=\rq\uq_1$.
From the first and third equations in \eqref{fluid system-}
\begin{align}
    \nx\cdot\bu=\nx\cdot(\rq\uq_1)=P\nx\cdot\left(\frac{\uq_1}{\tq}\right)=0,
\end{align}
we have 
\begin{align}\label{extra 05}
    \nx\cdot\uq_1=\uq_1\cdot\frac{\nx\tq}{\tq}.
\end{align}

From the second equation in \eqref{fluid system-} and \eqref{extra 05}, we have
\begin{align}\label{aa 15}
    -\frac{5}{3}\lambda[1]\Delta_xu_1+\nx\mathfrak{p}=&-\frac{5}{3}\big(\lambda[1]-\lambda[T]\big)\Delta_xu_1\\
    &-\nx\cdot\left(\frac{\lambda^2[T]}{P}\Big(K_1[T]\big(\nx^2T-\frac{1}{3}\dx T\id\big)+\frac{K_2[T]}{T}\big(\nx T\otimes\nx T-\frac{1}{3}\abs{\nx T}^2\id\big)\Big)\right)\no\\
    &+\nx\lambda[T]\cdot\left(\nx u_1+\big(\nx u_1\big)^t-\frac{2}{3}\big(\nx\cdot\uq_1\big)\id\right)+\lambda[T]\nx\left(\uq_1\cdot\frac{\nx\tq}{\tq}\right)-\frac{P}{T}u_1\cdot\nx u_1.\no
\end{align}
Hence, we know
\begin{align}\label{aa 15'}
    -\frac{5}{3P}\lambda[1]\Delta_x\bu+\nx \mathfrak{p}=&-\frac{5}{3P}\big(\lambda[1]-\lambda[T]\big)\Delta_x\bu+\frac{5}{3P}\lambda[T]\Delta_x\big((T-1)\bu\big)\\
    &-\nx\cdot\left(\frac{\lambda^2[T]}{P}\Big(K_1[T]\big(\nx^2T-\frac{1}{3}\dx T\id\big)+\frac{K_2[T]}{T}\big(\nx T\otimes\nx T-\frac{1}{3}\abs{\nx T}^2\id\big)\Big)\right)\no\\
    &+\nx\lambda[T]\cdot\left(\nx\left(P^{-1}T\bu\right)+\big(\nx\left(P^{-1}T\bu\right)\big)^t-\frac{2}{3}\big(\nx\cdot\left(P^{-1}T\bu\right)\big)\id\right)\no\\
    &+\lambda[T]\nx\left(P^{-1}\bu\cdot\nx\tq\right)-\bu\cdot\nx\left(P^{-1}T\bu\right).\no
\end{align}

Further, from the fourth equations in \eqref{fluid system-}
\begin{align}
    \nx\cdot\uq_1=\frac{1}{5P}\nx\cdot\left(\k\dfrac{\nx\tq}{2\tq^2}\right)=\frac{1}{5P}\frac{\k}{2T^2}\Delta_xT+\frac{1}{5P}\nx\left(\dfrac{\k}{2\tq^2}\right)\cdot \nx\tq,
\end{align}
we have
\begin{align}\label{aa 14}
    \frac{1}{5P}\frac{\k}{2T^2}\Delta_xT=P^{-1}\bu\cdot\nx\tq-\frac{1}{5P}\nx\left(\dfrac{\k}{2\tq^2}\right)\cdot \nx\tq.
\end{align}
Then we know 
\begin{align}\label{aa 16}
    \Delta_x\tq=&\frac{10T^2}{\k[T]}\big(\bu\cdot\nx\tq\big)-\frac{2T^2}{\k[T]}\nx\left(\dfrac{\k[T]}{2\tq^2}\right)\cdot \nx\tq.
\end{align}

\paragraph{\underline{Setup of Contraction Mapping}}
Collecting \eqref{extra 05}, \eqref{aa 15} and \eqref{aa 16}, this is a system for the pair $(\bu,T)$. Then we can design a mapping $W^{3,\NN}\times W^{4,\NN}\rt W^{3,\NN}\times W^{4,\NN}:(\widetilde{u},\widetilde{T})\rt(\bu,T)$
\begin{align}
\left\{
\begin{array}{rcl}
    -\frac{5}{3P}\lambda[1]\Delta_x\bu+\nx\mathfrak{p}&=&Z_1,\\\rule{0ex}{1.5em}
    \nx\cdot\bu&=&0,\\\rule{0ex}{1.5em}
    \Delta_x\tq&=&Z_3.
\end{array}
\right.
\end{align}
where
\begin{align}
    Z_1:=&-\frac{5}{3P}\Big(\lambda[1]-\lambda\big[\widetilde{T}\big]\Big)\Delta_x\widetilde{u}+\frac{5}{3P}\lambda\big[\widetilde{T}\big]\Delta_x\big((\widetilde{T}-1)\widetilde{u}\big)\\
    &-\nx\cdot\left(\frac{\lambda^2\big[\widetilde{T}\big]}{P}\Big(K_1\big[\widetilde{T}\big]\big(\nx^2\widetilde{T}-\frac{1}{3}\dx\widetilde{T}\id\big)+\frac{K_2}{\widetilde{T}}\big[\widetilde{T}\big]\big(\nx \widetilde{T}\otimes\nx \widetilde{T}-\frac{1}{3}\abs{\nx\widetilde{T}}^2\id\big)\Big)\right)\no\\
    &+\nx\lambda\big[\widetilde{T}\big]\cdot\left(\nx \big(P^{-1}\widetilde{T}\widetilde{u}\big)+\big(\nx \big(P^{-1}\widetilde{T}\widetilde{u}\big)\big)^t-\frac{2}{3}\big(\nx\cdot \big(P^{-1}\widetilde{T}\widetilde{u}\big)\big)\id\right)\no\\
    &+\lambda\big[\widetilde{T}\big]\nx\left(P^{-1}\widetilde{u}\cdot\nx\widetilde{T}\right)-\widetilde{u}\cdot\nx\big(P^{-1}\widetilde{T}\widetilde{u}\big),\no\\
    Z_3:=&\frac{10\widetilde{T}^2}{\k\big[\widetilde{T}\big]}\big(\widetilde{u}\cdot\nx\widetilde{T}\big)-\frac{2\widetilde{T}^2}{\k\big[\widetilde{T}\big]}\nx\left(\dfrac{\k\big[\widetilde{T}\big]}{2\widetilde{T}^2}\right)\cdot \nx\widetilde{T}.
\end{align}

\paragraph{\underline{Boundedness and Contraction}}
Based on \cite{Cattabriga1961} and \cite[Theorem IV.5.8]{Boyer.Fabrie2013}, noticing the compatibility condition
\begin{align}
    \int_{\p\Omega}\bu\cdot n=\int_{\Omega}(\nx\cdot\bu)=0,
\end{align}
we know that 
\begin{align}
    \nm{\bu}_{W^{3,\NN}}+\nm{\mathfrak{p}}_{W^{2,\NN}}\ls \nm{Z_1}_{W^{1,\NN}}+\abs{\bu}_{W^{3-\frac{1}{\NN},\NN}}.
\end{align}
Based on standard elliptic estimates \cite{Krylov2008}, we have
\begin{align}
    \nm{T-1}_{W^{4,\NN}}\ls\nm{Z_3}_{W^{2,\NN}}+\abs{T}_{W^{4-\frac{1}{\NN},\NN}}.
\end{align}
Under the assumption
\begin{align}
    \nm{\widetilde{u}_1}_{W^{3,\NN}}+\nm{\widetilde{T}-1}_{W^{4,\NN}}\ls2\oot,
\end{align}
we directly obtain
\begin{align}
    \nm{Z_1}_{W^{1,\NN}}\ls&\oot\Big(\nm{\widetilde{u}}_{W^{3,\NN}}+\nm{\nx\widetilde{T}}_{W^{3,\NN}}\Big),\\
    \nm{Z_3}_{W^{2,\NN}}\ls&\oot\Big(\nm{\widetilde{u}}_{W^{3,\NN}}+\nm{\nx\widetilde{T}}_{W^{3,\NN}}\Big).
\end{align}
Hence, we know that
\begin{align}
    \nm{\bu}_{W^{3,\NN}}+\nm{\mathfrak{p}}_{W^{2,\NN}}+\nm{T-1}_{W^{4,\NN}}\ls\oot\Big(\nm{\widetilde{u}}_{W^{3,\NN}}+\nm{\nx\widetilde{T}}_{W^{3,\NN}}\Big)+\as\leq2\oot.
\end{align}
Hence, this mapping is bounded. 

By a similar argument, for $(\widetilde{u}^{[k]},\widetilde{T}^{[k]})\rt(\bu^{[k]},T^{[k]})$ with $k=1,2$, we can show that
\begin{align}
    &\nm{\bu^{[1]}-\bu^{[2]}}_{W^{3,\NN}}+\nm{\mathfrak{p}^{[1]}-\mathfrak{p}^{[2]}}_{W^{2,\NN}}+\nm{T^{[1]}-T^{[2]}}_{W^{4,\NN}}\\
    \ls&\Big(\nm{\widetilde{u}}_{W^{3,\NN}}+\nm{\nx\widetilde{T}}_{W^{3,\NN}}\Big)\Big(\nm{\widetilde{u}^{[1]}-\widetilde{u}^{[2]}}_{W^{3,\NN}}+\nm{\nx\widetilde{T}^{[1]}-\nx\widetilde{T}^{[2]}}_{W^{3,\NN}}\Big),\no
\end{align}
which yields
\begin{align}
    &\nm{\bu^{[1]}-\bu^{[2]}}_{W^{3,\NN}}+\nm{\mathfrak{p}^{[1]}-\mathfrak{p}^{[2]}}_{W^{2,\NN}}+\nm{T^{[1]}-T^{[2]}}_{W^{4,\NN}}\\
    \ls& \oot\Big(\nm{\widetilde{u}^{[1]}-\widetilde{u}^{[2]}}_{W^{3,\NN}}+\nm{\nx\widetilde{T}^{[1]}-\nx\widetilde{T}^{[2]}}_{W^{3,\NN}}\Big).\no
\end{align}
Hence, this is a contraction mapping.

In summary, we know that there exists a unique solution to \eqref{fluid system-} satisfying 
\begin{align}
    \nm{\bu}_{W^{3,\NN}}+\nm{\mathfrak{p}}_{W^{2,\NN}}+\nm{T-1}_{W^{4,\NN}}\ls\oot,
\end{align}
and further
\begin{align}
    \nm{u_1}_{W^{3,\NN}}+\nm{\mathfrak{p}}_{W^{2,\NN}}+\nm{T-1}_{W^{4,\NN}}\ls\oot.
\end{align}

\begin{remark}
    Based on the first equation in \eqref{aa 22}, we have
    \begin{align}
        \rho=PT^{-1}\in W^{4,\NN}.
    \end{align}
    Then we have
    \begin{align}\label{aa 33}
        P\abs{\Omega}=\int_{\Omega}\rq(x)\tq(x)\ud x=\frac{1}{3}\iint_{\Omega\times\r^3}\abs{v}^2\m(x,v)\ud v\ud x.
    \end{align}
\end{remark}

\end{proof}

\subsection{Construction of Boundary Layer $\fb_1$ - Stage II}

Now we can define the full boundary layer.
Define 
\begin{align}
    \blf_{n}:=\big(\p_{n}\tq\big)\mathcal{H}^{(n)},
\end{align}
where $\mathcal{H}^{(n)}$ solves the Milne problem
\begin{align}\label{extra 14}
\left\{
\begin{array}{l}
    \va\dfrac{\p\mathcal{H}^{(n)}}{\p\eta}+\lc_w\left[\mathcal{H}^{(n)}\right]=0,\\
    \mathcal{H}^{(n)}(0,\vvv)=-\dfrac{\a\cdot n}{2T^2}\ \ \text{for}\ \ \va>0,\\
    \lim_{\eta\rt\infty}\mathcal{H}^{(n)}(\eta,\vvv)=\mathcal{H}^{(n)}_{\infty}\in\nk,
    \end{array}
    \right.
\end{align}
with the zero mass-flux condition
\begin{align}
    \int_{\r^3}\va\mbh(\vvv)\mathcal{H}^{(n)}=0.
\end{align}
Denote
\begin{align}
    \blf_{n,\infty}:=\big(\p_{n}\tq\big)\mathcal{H}^{(n)}_{\infty}.
\end{align}
Here $\p_n\tq$ comes from the ghost-effect equation \eqref{fluid system-} and is well-defined due to Theorem \ref{thm:ghost}.

Finally, we have the full boundary layer from \eqref{extra 15}:
\begin{align}\label{extra 16}
    \fb_1(\eta,\vvv)=&\ch\left(\e^{-1}\va\right)\chi(\e\eta)\Big(\blf_{\iota_1}(\eta,\vvv)+\blf_{\iota_2}(\eta,\vvv)+\blf_{n}(\eta,\vvv)-\blf_{\iota_1,\infty}-\blf_{\iota_2,\infty}-\blf_{n,\infty}\Big)\\
    =&\ch\left(\e^{-1}\va\right)\chi(\e\eta)\Big(\blff_{\iota_1}(\eta,\vvv)+\blff_{\iota_2}(\eta,\vvv)+\blff_{n}(\eta,\vvv)\Big).\no
\end{align}

Since the cutoff in $\fb_1$ is only defined in the normal direction, we can deduce tangential regularity estimates from Theorem \ref{boundary well-posedness}:
\begin{theorem}\label{thm:boundary}
    Under the assumption \eqref{assumption:boundary}, we can construct $\fb_1$ such that for $i=1,2$, some $K_0>0$ and any $0<\N\leq3$
    \begin{align}\label{final 38}
        \lnm{\ue^{K_0\eta}\fb_1}+\lnm{\ue^{K_0\eta}\frac{\p^{\N}\fb_1}{\p\iota_i^{\N}}}\ls\oot.
    \end{align}
\end{theorem}

From \eqref{aa 24} and \eqref{extra 17}, this fully determines the boundary condition of $\tq_1$:
\begin{align}\label{aa 24'}
    T_1(x_0)=T^B.
\end{align}

\subsection{Construction of $(\rq_1,\tq_1)$}

\begin{theorem}\label{cor:high1}
    Under the assumption \eqref{assumption:boundary}, we can construct $(\rq_1,\tq_1)$ such that for any $\NN\in[2,\infty)$
    \begin{align}
    \nm{f_1}_{W^{3,\NN}L^{\infty}_{\varrho,\vartheta}}+\abs{f_1}_{W^{3-\frac{1}{\NN},\NN}L^{\infty}_{\varrho,\vartheta}}&\ls\oot.
    \end{align}
\end{theorem}

\begin{proof}
The boundary condition in \eqref{aa 24} and Theorem \ref{thm:boundary} imply that 
\begin{align}
    \abs{T_1}_{W^{3,\NN}}\ls\oot.
\end{align}
Then we can freely define a Sobolev extension for $\tq_1$ such that
\begin{align}
    \nm{T_1}_{W^{3+\frac{1}{\NN},\NN}}\ls\oot.
\end{align}
We choose the constant 
\begin{align}
    P_1=0.
\end{align}
Then we can deduce that
{\begin{align}\label{aa 34}
    &\iint_{\Omega\times\r^3}\abs{v}^2\big(\mh\f_1+\mbh\fb_1+\e\mh\f_2(x,v)\big)\ud x\ud v\\
    =&\int_{\Omega}\big(3\rq_1(x)\tq(x)+3\tq_1(x)\rq(x)+3\e\rq_2(x)\tq(x)+3\e\rq(x)\tq_2(x)\big)\ud x+\iint_{\Omega\times\r^3}\abs{v}^2\mbh\fb_1\ud x\ud v\no\\
    =&\int_{\Omega}\Big(3\rq_1(x)\tq(x)+3\tq_1(x)\rq(x))\ud x+\iint_{\Omega\times\r^3}\abs{v}^2\mbh\fb_1\ud x\ud v\no\\
    =&\int_{\Omega}3P_1\ud x+\iint_{\Omega\times\r^3}\abs{v}^2\mbh\fb_1\ud x\ud v=\iint_{\Omega\times\r^3}\abs{v}^2\mbh\fb_1\ud x\ud v,\no
\end{align}}
where we have used $\ds\int_{\Omega}\mathfrak{p}=\int_{\Omega}\Big(T\rho_2+\rho T_2\Big)=0$.

Then based on \eqref{fluid system'}, we have
\begin{align}\label{aa 30}
    \rho_1=-T^{-1}\big(\rho T_1\big),
\end{align}
and thus
{\begin{align}
    \nm{\rq_1}_{W^{3+\frac{1}{\NN},\NN}}\ls\oot.
\end{align}}
Note that $\rho_1$ is not necessarily equal to $\rho^B$ on $\p\Omega$. However, \eqref{aa 10} can still hold due to \eqref{aa 02}.

Hence, we have shown that
\begin{align}
    \nm{f_1}_{W^{3,\NN}L^{\infty}_{\varrho,\vartheta}}+\abs{f_1}_{W^{3-\frac{1}{\NN},\NN}L^{\infty}_{\varrho,\vartheta}}\ls\nm{\rho_1}_{W^{3,\NN}}+\nm{u_1}_{W^{3,\NN}}+\nm{T_1}_{W^{3,\NN}}+\nm{T}_{W^{3,\NN}} \ls\oot.
\end{align}
\end{proof}

\begin{remark}
    We assume that the remainder $\re$ satisfies 
    \begin{align}
    \iint_{\Omega\times\r^3}\abs{v}^2\mh\re(x,v)\ud x\ud v=0.
    \end{align}
    Hence, combining \eqref{aa 33}, \eqref{aa 34} and \eqref{aa 08}, we know 
    \begin{align}
    \iint_{\Omega\times\r^3}\abs{v}^2\fs(x,v)\ud v\ud x=\iint_{\Omega\times\r^3}\abs{v}^2\m(x,v)\ud v\ud x=3P\abs{\Omega}+\e\iint_{\Omega\times\r^3}\abs{v}^2\mbh\fb_1\ud x\ud v.
    \end{align}
\end{remark}

\subsection{Construction of $(\rq_2,\uq_2,\tq_2)$}

\begin{theorem}\label{cor:high2}
    Under the assumption \eqref{assumption:boundary}, we can construct $(\rq_2,\uq_2,\tq_2)$ such that for any $\NN\in[2,\infty)$
    \begin{align}
    \nm{f_2}_{W^{2,\NN}L^{\infty}_{\varrho,\vartheta}}+\abs{f_2}_{W^{2-\frac{1}{\NN},\NN}L^{\infty}_{\varrho,\vartheta}}&\ls\oot.
    \end{align}
\end{theorem}

\begin{proof}
Denote
\begin{align}
    Y(\iota_1,\iota_2):=-\e^{-1}P^{-1}\int_{\r^3}\va\mbh(\vvv)\fb_1(0,\vvv)\ud \vvv.
\end{align}
Due to \eqref{aa 37}, we have $\abs{Y}\ls \oot$. 
Then we define $u_2$ via $u_2=\nx\psi$ where $\psi$ solves
\begin{align}
    \left\{
    \begin{array}{ll}
        -\Delta_x\psi=\ds-\abs{\Omega}^{-1}\int_{\p\Omega}Y(\iota_1,\iota_2)\ud s&\ \ \text{in}\ \ \Omega,\\\rule{0ex}{1.5em}
        \dfrac{\p\psi}{\p n}=Y&\ \ \text{on}\ \ \p\Omega.
    \end{array}
    \right.
\end{align}
Due to classical elliptic theory, we know that this equation is well-posed. In particular, due to \eqref{final 38}, we know $Y\in W^{3,\infty}(\p\Omega)$. Then we have $\psi\in W^{4,\NN}$ and thus $u_2\in W^{3,\NN}$ satisfying
\begin{align}
    \nm{u_2}_{W^{3,\NN}}\ls\oot.
\end{align}

From Theorem \ref{thm:ghost} and the third equation in \eqref{aa 22}, we know that
\begin{align}
    T\rho_2+\rho T_2\in W^{2,\NN}.
\end{align}
We are free to take $\rho_2=0$ in $\Omega$, and thus $T_2$ is determined and satisfies
{\begin{align}
    \nm{T_2}_{W^{2,\NN}}\ls\oot.
\end{align}}

Hence, we have shown that
\begin{align}
    \nm{f_2}_{W^{2,\NN}L^{\infty}_{\varrho,\vartheta}}+\abs{f_2}_{W^{2-\frac{1}{\NN},\NN}L^{\infty}_{\varrho,\vartheta}}\ls\nm{f_1}_{W^{2,\NN}L^{\infty}_{\varrho,\vartheta}}+\nm{\rho_2}_{W^{2,\NN}}+\nm{u_2}_{W^{2,\NN}}+\nm{T_2}_{W^{2,\NN}}\ls\oot.
\end{align}
\end{proof}

\begin{remark}
    Such choice of $u_2$ implies that on the boundary $\p\Omega$
    \begin{align}
    u_2\cdot n=Y.
    \end{align}
    Hence, we know
    \begin{align}\label{aa 38}
    \int_{\r^3}\big(\e^2\f_2+\e\fb_1\big)\mb(v\cdot n)=\e^2P(u_2\cdot n)+\e\int_{\r^3}\va\mbh(\vvv)\fb_1(0,\vvv)\ud \vvv=0,
    \end{align}
    and thus
    \begin{align}
    \int_{\r^3}\big(\mh+\e\f_1+\e^2\f_2+\e\fb_1\big)\mbh(v\cdot n)=0.
    \end{align}
    We restore the zero mass-flux condition of $\mh+\e\f_1+\e^2\f_2+\e\fb_1$.
\end{remark}

\section{Remainder Equation}

For sake of completeness, in this section we will present the remainder equation for $\re$ and report the main result in \cite{AA023}.

Now we begin to derive the remainder equation for $\re$ in \eqref{aa 08}, or equivalently the nonlinear Boltzmann equation \eqref{large system}. 
Denote 
\begin{align}
    Q[F,F]=&Q_{\text{gain}}[F,F]-Q_{\text{loss}}[F,F]\\
    :=&\int_{\r^3}\int_{\s^2}q(\vo,\abs{\vuu-\vv})F(\vuu_{\ast})F(\vv_{\ast})\ud{\vo}\ud{\vuu}-F(\vv)\int_{\r^3}\int_{\s^2}q(\vo,\abs{\vuu-\vv})F(\vuu)\ud{\vo}\ud{\vuu}=\nu(F)F.\no
\end{align}
Denote $\fs=\ff+\e^{\al}\mh\re$,
where
\begin{align}
    \ff:=\m+\mh\big(\e\f_1+\e^2\f_2\big)+\mbh\big(\e\fb_1\big).
\end{align}
We can split $\fs=\fs_+-\fs_-$ where $\fs_+=\max\{\fs,0\}$ and $\fs_-=\max\{-\fs,0\}$ denote the positive and negative parts, and the similar notation also applies to $\ff$ and $\re$. 

In order to study \eqref{large system}, we first consider an auxiliary equation (which is equivalent to \eqref{large system} when $\fs\geq0$)
\begin{align}\label{auxiliary system}
    \vv\cdot\nx \fs+\e^{-1}\Big(Q_{\text{loss}}[\fs,\fs]-Q_{\text{gain}}[\fs_+,\fs_+]\Big)=\mathfrak{z}\ds\iint_{\Omega\times\r^3}\e^{-1}\Big(Q_{\text{loss}}[\fs,\fs]-Q_{\text{gain}}[\fs_+,\fs_+]\Big),
\end{align}
with diffuse-reflection boundary condition
\begin{align}
    \fs(\vx_0,\vv)=\ms(\vx_0,\vv)\displaystyle\int_{\vv'\cdot\vn(\vx_0)>0}
    \fs(\vx_0,\vv')\abs{\vv'\cdot\vn(\vx_0)}\ud{\vv'} \ \ \text{for}\ \ \vx_0\in\p\Omega\ \ \text{and}\ \ \vv\cdot\vn(\vx_0)<0.
\end{align}
Here $\mathfrak{z}=\mathfrak{z}(v)>0$ is a smooth function with support contained in $\{\abs{v}\leq 1\}$ such that $\ds\iint_{\Omega\times\r^3}\mathfrak{z}=1$.

The auxiliary system \eqref{auxiliary system} is equivalent to
\begin{align}
\vv\cdot\nx \fs-\e^{-1}Q[\fs,\fs]
=&-\e^{-1}\Big(Q_{\text{gain}}[\fs,\fs]-Q_{\text{gain}}[\fs_+,\fs_+]\Big)+\mathfrak{z}\ds\iint_{\Omega\times\r^3}\e^{-1}\Big(Q_{\text{loss}}[\fs,\fs]-Q_{\text{gain}}[\fs_+,\fs_+]\Big),
\end{align} 
and due to orthogonality of $Q$, is further equivalent to
\begin{align}\label{auxiliary system'}
\vv\cdot\nx \fs-\e^{-1}Q[\fs,\fs]
=&-\e^{-1}\Big(Q_{\text{gain}}[\fs,\fs]-Q_{\text{gain}}[\fs_+,\fs_+]\Big)+\mathfrak{z}\ds\iint_{\Omega\times\r^3}\e^{-1}\Big(Q_{\text{gain}}[\fs,\fs]-Q_{\text{gain}}[\fs_+,\fs_+]\Big).
\end{align}

\begin{remark}
The extra terms
\begin{align}
    -\e^{-1}\Big(Q_{\text{gain}}[\fs,\fs]-Q_{\text{gain}}[\fs_+,\fs_+]\Big)+\mathfrak{z}\ds\iint_{\Omega\times\r^3}\e^{-1}\Big(Q_{\text{gain}}[\fs,\fs]-Q_{\text{gain}}[\fs_+,\fs_+]\Big).
\end{align}
on the RHS of \eqref{auxiliary system'} plays a significant role in justifying the positivity of $\fs$ (see \cite{AA023}).
Clearly, when $\fs\geq0$, i.e. $\fs=\fs_+$, the above extra terms vanish and the auxiliary equation \eqref{auxiliary system'} reduces to \eqref{large system}.
\end{remark}

Inserting $\fs=\ff+\e^{\al}\mh\re:=\m+\ffe+\e^{\al}\mh\re$ into \eqref{auxiliary system'}, we have
\begin{align}\label{wt 02}
    &\vv\cdot\nx \left(\mh\re\right)+\e^{-1}\mh\lc[\re]\\
    =&\ \mathscr{\ss}+\e^{-1}\bigg(2Q^{\ast}\left[\ffe,\mh\re\right]+\e^{\al}Q^{\ast}\left[\mh\re,\mh\re\right]\bigg)\no\\
    &-\e^{-\al}\bigg(\e^{-1}Q_{\text{gain}}\left[\ff+\e^{\al}\mh\re,\ff+\e^{\al}\mh\re\right]-\e^{-1}Q_{\text{gain}}\left[\left(\ff+\e^{\al}\mh\re\right)_+,\left(\ff+\e^{\al}\mh\re\right)_+\right]\bigg)\no\\
    &+\e^{-\al}\mathfrak{z}\ds\iint_{\Omega\times\r^3}\bigg(\e^{-1}Q_{\text{gain}}\left[\ff+\e^{\al}\mh\re,\ff+\e^{\al}\mh\re\right]-\e^{-1}Q_{\text{gain}}\left[\left(\ff+\e^{\al}\mh\re\right)_+,\left(\ff+\e^{\al}\mh\re\right)_+\right]\bigg),\no
\end{align}
where 
\begin{align}\label{extra 37}
    \mathscr{\ss}:=-\e^{-\al}v\cdot\nx\ff+\e^{-\al-1}Q^{\ast}\left[\ff,\ff\right].
\end{align}
Hence, we know that the equation for the remainder $\re$ is
{\begin{align}\label{remainder}
\left\{
\begin{array}{l}
\vv\cdot\nx\left(\mh\re\right)+\e^{-1}\mh\lc[\re]=\mh\ss\ \ \text{in}\ \
\Omega\times\r^3,\\\rule{0ex}{1.5em} \re(\vx_0,\vv)=\pp[\re](\vx_0,\vv)+h(\vx_0,\vv) \ \ \text{for}\ \ \vx_0\in\p\Omega\ \ \text{and}\ \ \vv\cdot\vn(\vx_0)<0.
\end{array}
\right.
\end{align}}
where
\begin{align}
\pp[\re](\vx_0,\vv):=\mss(\vx_0,\vv)\displaystyle\int_{\vv'\cdot\vn(\vx_0)>0}
\mh(\vx_0,\vv')\re(\vx_0,\vv')\abs{\vv'\cdot\vn(\vx_0)}\ud{\vv'},
\end{align}
with 
\begin{align}
\mss(\vx_0,\vv):=\ms\mb^{-\frac{1}{2}},
\end{align}
satisfying the normalization condition
\begin{align}
    \mb^{\frac{1}{2}}(\vx_0,\vv)=\mss(\vx_0,\vv)\displaystyle\int_{\vv'\cdot\vn(\vx_0)>0}
    \mb(\vx_0,\vv')\abs{\vv'\cdot\vn(\vx_0)}\ud{\vv'}=\frac{P}{\big(2\pi T_w(x_0)\big)^{\frac{1}{2}}}\mss(\vx_0,\vv).
\end{align}
The source term $\ss$ includes the nonlinear terms and the terms of the expansion coming from higher orders and $h$ is a correction on the boundary condition. 

\begin{lemma}
    We have
    \begin{align}\label{aa 32}
    h:=&\e^{-\al}\left(\pp\left[\mhh\ffe\right]-\mhh\ffe\right)\\
    =&\e^{2-\al}\left(\mss\displaystyle\int_{\vv'\cdot\vn>0}
    \mbh(v')\f_2(v')\abs{\vv'\cdot\vn}\ud{\vv'}-\f_2\Big|_{v\cdot n<0}\right)\no\\
    &-\e^{1-\al}\left(\mss\displaystyle\int_{\vv'\cdot\vn>0}
    \mbh\chi\left(\e^{-1}\va\right)\blff\abs{\vv'\cdot\vn}\ud{\vv'}-\mbh\chi\left(\e^{-1}\va\right)\blff\Big|_{v\cdot n<0}\right).\no
\end{align}
\end{lemma}
\begin{proof}
From \eqref{aa 08}, we know
\begin{align}
    h:=&\e^{-\al}\left(\pp\left[\mhh\ffe\right]-\mhh\ffe\right).
\end{align}
Then due to \eqref{aa 10} and \eqref{aa 24}, we know
\begin{align}
    &\e^{-\al}\left(\pp\left[\mhh\Big(\e\f_1+\e\fb_1\Big)\right]-\mhh\Big(\e\f_1+\e\fb_1\Big)\right)\\
    =&\e^{1-\al}\left(\pp\left[\mhh\Big(\f_1+\blff-\chi\big(\e^{-1}\va\big)\blff\Big)\right]-\mhh\Big(\f_1+\blff-\chi\big(\e^{-1}\va\big)\blff\Big)\right)\no\\
    =&-\e^{1-\al}\left(\mss\displaystyle\int_{\vv'\cdot\vn>0}
    \mbh\chi\left(\e^{-1}\va\right)\blff\abs{\vv'\cdot\vn}\ud{\vv'}-\mbh\chi\left(\e^{-1}\va\right)\blff\Big|_{v\cdot n<0}\right).\no
\end{align}
Then the result  
follows by adding the $\f_2$ contribution.
\end{proof}

\begin{lemma}
We have
\begin{align}
    \ss:=&\ \mhh\mathscr{\ss}+\e^{-1}\mhh\bigg(2Q^{\ast}\left[\ffe,\mh\re\right]+\e^{\al}Q^{\ast}\left[\mh\re,\mh\re\right]\bigg)\\
    &-\e^{-\al}\mhh\bigg(\e^{-1}Q_{\text{gain}}\left[\ff+\e^{\al}\mh\re,\ff+\e^{\al}\mh\re\right]-\e^{-1}Q_{\text{gain}}\left[\left(\ff+\e^{\al}\mh\re\right)_+,\left(\ff+\e^{\al}\mh\re\right)_+\right]\bigg)\no\\
    &+\e^{-\al}\mathfrak{z}\mhh\iint_{\Omega\times\r^3}\bigg(\e^{-1}Q_{\text{gain}}\left[\ff+\e^{\al}\mh\re,\ff+\e^{\al}\mh\re\right]-\e^{-1}Q_{\text{gain}}\left[\left(\ff+\e^{\al}\mh\re\right)_+,\left(\ff+\e^{\al}\mh\re\right)_+\right]\bigg),\no
\end{align}
where $\mathscr{\ss}$ is defined in \eqref{extra 37}.
The detailed expression is
\begin{align}\label{aa 31}
    \ss=-\llc[\re]+\sb,
\end{align}
where
\begin{align}
    \llc[\re]:=-2\e^{-1}\mhh Q^{\ast}\left[\mh\Big(\e\f_1\Big),\mh\re\right]=-2\Gamma[\f_1,\re],
\end{align}
\begin{align}
    \sb:=\ss_0+\ss_1+\ss_2+\ss_3+\ss_4+\ss_5+\sp,
\end{align}
for
\begin{align}
    \ss_0:=&2\e^{-1}\mhh Q^{\ast}\left[\mh\Big(\e^2\f_2\Big),\mh\re\right]=2\e\Gamma[f_2,\re],\\
    \ss_1:=&2\e^{-1}\mhh Q^{\ast}\left[\mh_w\Big(\e\fb_1\Big),\mh\re\right]=2\Gamma\left[\mhh\mh_w\fb_1,\re\right],\\
    \ss_2:=&\e^{\alpha-1}\mhh Q^{\ast}\left[\mh\re,\mh\re\right]=\e^{\alpha-1}\Gamma[\re,\re],
\end{align}
\begin{align}
    \ss_3:=&\e^{1-\alpha}\mhh\dfrac{1}{R_1-\e\eta}\bigg(\vb^2\dfrac{\p }{\p\va}-\va\vb\dfrac{\p}{\p\vb}\bigg)\left(\mbh\fb_1\right)\label{mm 00}\\
    &+\e^{1-\alpha}\mhh\dfrac{1}{R_2-\e\eta}\bigg(\vc^2\dfrac{\p }{\p\va}-\va\vc\dfrac{\p }{\p\vc}\bigg)\left(\mbh\fb_1\right)\no\\
    &-\e^{1-\alpha}\mhh\dfrac{1}{\pl_1\pl_2}\left(\dfrac{R_1\p_{\iota_1\iota_1}\vr\cdot\p_{\iota_2}\vr}{\pl_1(R_1-\e\eta)}\vb\vc
    +\dfrac{R_2\p_{\iota_1\iota_2}\vr\cdot\p_{\iota_2}\vr}{\pl_2(R_2-\e\eta)}\vc^2\right)\dfrac{\p }{\p\vb}\left(\mbh\fb_1\right)\no\\
    &-\e^{1-\alpha}\mhh\dfrac{1}{\pl_1\pl_2}\left(\dfrac{R_2\p_{\iota_2\iota_2}\vr\cdot\p_{\iota_1}\vr}{\pl_2(R_2-\e\eta)}\vb\vc
    +\dfrac{R_1\p_{\iota_1\iota_2}\vr\cdot\p_{\iota_1}\vr}{\pl_1(R_1-\e\eta)}\vb^2\right)\dfrac{\p }{\p\vc}\left(\mbh\fb_1\right)\no\\
    &-\e^{1-\alpha}\mhh\left(\dfrac{R_1\vb}{\pl_1(R_1-\e\eta)}\dfrac{\p }{\p\iota_1}+\dfrac{R_2\vc}{\pl_2(R_2-\e\eta)}\dfrac{\p }{\p\iota_2}\right)\left(\mbh\fb_1\right)\no\\
    &+\e^{-\alpha}\mhh\va\ch(\e^{-1}\va)\frac{\p\chi(\e\eta)}{\p\eta}\left(\mbh\blff\right)+\e^{-\alpha}\mhh\mbh\chi(\e\eta)\bigg(\ch(\e^{-1}\va)K_w\Big[\blff\Big]-K_w\left[\ch(\e^{-1}\va)\blff\right]\bigg),\no\\
    \ss_4:=&-\e^{-\alpha}\mhh\left(\vv\cdot\nx\left(\mh\left(\e^2\f_2\right)\right)\right)=-\e^{2-\alpha}\mhh\left(\vv\cdot\nx\left(\mh\f_2\right)\right),\\
    \ss_5:=&\e^{3-\alpha}\mhh Q^{\ast}\left[\mh\f_2,\mh\f_2\right]+2\e^{2-\alpha}\mhh Q^{\ast}\left[\mh\f_2,\mh\f_1\right]+2\e^{2-\alpha}\mhh Q^{\ast}\left[\mh\f_2,\mbh\fb_1\right]\\
    &+2\e^{1-\alpha}\mhh Q^{\ast}\left[\mbh\fb_1,\mh\f_1\right]+\e^{1-\alpha}\mhh Q^{\ast}\left[\mbh\fb_1,\mbh\fb_1\right]+\e^{-\alpha}\mhh Q^{\ast}\left[\m-\mb,\mbh\fb_1\right]\no\\
    =&\e^{3-\alpha}\Gamma\left[\f_2,\f_2\right]+2\e^{2-\alpha}\Gamma\left[\f_2,\f_1\right]+2\e^{2-\alpha}\Gamma\left[\f_2,\mhh\mbh\fb_1\right]\no\\
    &+2\e^{1-\alpha}\Gamma\left[\mhh\mbh\fb_1,\f_1\right]+\e^{1-\alpha}\Gamma\left[\mhh\mbh\fb_1,\mhh\mbh\fb_1\right]+\e^{-\alpha}\Gamma\left[\mhh(\m-\mb),\mhh\mbh\fb_1\right]\no,
\end{align}
and
\begin{align}
    \sp:=&-\e^{-\al}\mhh\bigg(\e^{-1} Q_{\text{gain}}\left[\ff+\e^{\al}\mh\re,\ff+\e^{\al}\mh\re\right]-\e^{-1} Q_{\text{gain}}\left[\left(\ff+\e^{\al}\mh\re\right)_+,\left(\ff+\e^{\al}\mh\re\right)_+\right]\bigg)\\
    &+\e^{-\al}\mathfrak{z}\mhh\ds\iint_{\Omega\times\r^3}\bigg(\e^{-1} Q_{\text{gain}}\left[\ff+\e^{\al}\mh\re,\ff+\e^{\al}\mh\re\right]-\e^{-1} Q_{\text{gain}}\left[\left(\ff+\e^{\al}\mh\re\right)_+,\left(\ff+\e^{\al}\mh\re\right)_+\right]\bigg).\no
\end{align}
\end{lemma}
\begin{proof}
    This follows directly from \eqref{wt 02}.
\end{proof}

We decompose
\begin{align}
    \re=&\pk[\re]+(\ik-\pk)[\re]
    :=\mh(\vv)\bigg(\P_{\re}(\vx)+\vv\cdot
    \bb_{\re}(\vx)+\Big(\abs{\vv}^2-5\tq\Big)c_{\re}(\vx)\bigg)+(\ik-\pk)[\re],
\end{align}
We further define the orthogonal split
\begin{align}
(\ik-\pk)[{\re}]
&=\a\cdot\bd_{\re}(\vx)+(\ik-\bpk)[{\re}],
\end{align}
where $(\ik-\bpk)[{\re}]$ is the orthogonal complement to $\a\cdot\bd_{\re}(\vx)$ in $\nnk$ with respect to $\brr{\ \cdot\ ,\ \cdot\ }=\brr{\ \cdot\ ,\lc[\ \cdot\ ]}$, i.e.
\begin{align}
\brr{\a,(\ik-\bpk)[{\re}]}=\brv{\ab,(\ik-\bpk)[{\re}]}=0.
\end{align}

In summary, we decompose the remainder as \eqref{pp 01},
{\begin{align}
&=\bigg(\P+\bb\cdot v+c\Big(\abs{v}^2-5T\Big)\bigg)\mh+\bd\cdot\a+(\ik-\bpk)[\re].\label{pp 01}
\end{align}}
We can further define the Hodge decomposition $\bd=\nx\xi+\be$
with $\xi$ solving the Poisson equation
\begin{align}\label{tt 01}
\left\{
\begin{array}{ll}
\nx\cdot\left(\k\nx\xi\right)=\nx\cdot(\k\bd)\ \ \text{in}\ \ \Omega,\\\rule{0ex}{1.2em}
\xi=0\ \ \text{on}\ \ \p\Omega.
\end{array}
\right.
\end{align}
We reformulate the remainder equation with a global Maxwellian in order to obtain $L^{\infty}$ estimates. Considering $\lnm{\nx\tq}\ls\oot$ for $\oot$ defined in \eqref{def:oot}, choose a constant $\tm$ such that
\begin{align}\label{def:tm}
    \tm<\min_{x\in\Omega}T<\max_{x\in\Omega}T<2\tm\ \ \text{and}\ \ \max_{x\in\Omega}T-\tm=\oot.
\end{align}
Define a global Maxwellian
\begin{align}\label{final 13}
    \mm:=\frac{P}{(2\pi)^{\frac{3}{2}}\tm^{\frac{5}{2}}}\exp\bigg(-\frac{\abs{v}^2}{2\tm}\bigg).
\end{align}
We can rewrite \eqref{remainder} as
\begin{align}\label{ll 01}
\left\{
\begin{array}{l}
\vv\cdot\nx\rem+\e^{-1}\lc_M[\re]=\ss_M\ \ \text{in}\ \
\Omega\times\r^3,\\\rule{0ex}{1.5em} \rem(\vx_0,\vv)=\mathcal{P}_M[\rem](\vx_0,\vv)+h_M(\vx_0,\vv) \ \ \text{for}\ \ \vx_0\in\p\Omega\ \ \text{and}\ \ \vv\cdot\vn(\vx_0)<0,
\end{array}
\right.
\end{align}
where $\rem=\mmhh\mh\re$, $\ss_M=\mmhh\mh\ss$, $h_M=\mmhh\mh h$ and for $\mmss:=\mmhh\mh(\vx_0,\vv)\mss(\vx_0,\vv)=M_w\mmhh$
\begin{align}
    \lc_M[\rem]:=&-2\mmhh Q\left[\m,\mmh \rem\right]:=\nu_M\rem-K_M[\rem],\\
    \mathcal{P}_M[\re_M](\vx_0,\vv):=&\mmss(\vx_0,\vv)\displaystyle\int_{\vv'\cdot\vn(\vx_0)>0}   \mmh\rem(\vx_0,\vv')\abs{\vv'\cdot\vn(\vx_0)}\ud{\vv'}.
\end{align}

Denote the working space $X$ via the norm
\begin{align}\label{ss 00}
\xnm{\re}:=&\e^{-1}\tnm{\P}+\e^{-\frac{1}{2}}\tnm{\bb}+\pnm{c}{2}+\e^{-1}\tnm{\xi}+\e^{-\frac{1}{2}}\nm{\xi}_{H^2}+\e^{-1}\tnm{\be}+\e^{-1}\um{(\ik-\bpk)[\re]}\\
&+\pnm{\P}{6}+\pnm{\bb}{6}+\pnm{c}{6}+\e^{-1}\pnm{\xi}{6}+\nm{\xi}_{W^{2,6}}+\pnm{\be}{6}+\pnm{(\ik-\bpk)[\re]}{6}\no\\
&+\tnms{\pp[\re]}{\gamma}+\e^{-\frac{1}{2}}\tnms{(1-\pp)[\re]}{\gamma_+}+\pnms{\m^{\frac{1}{4}}(1-\pp)[\re]}{4}{\gamma_+}+\e^{-\frac{1}{2}}\tnms{\nx\xi}{\p\Omega}\no\\
&+\e^{\frac{1}{2}}\lnmm{\rem}+\e^{\frac{1}{2}}\lnmms{\rem}{\gamma}.\no
\end{align}

In the companion paper \cite{AA023}, we prove the following:
\begin{theorem}
Assume that $\Omega$ is a bounded $C^3$ domain and \eqref{assumption:boundary} holds. Then for any given $P>0$, there exists $\e_0>0$ such that for any $\e\in(0,\e_0)$, there exists a nonnegative solution $\fs$ to the equation \eqref{large system-} represented by \eqref{aa 08} with $\al=1$ satisfying
\begin{align}\label{final 31}
    \int_{\Omega}\P(x)\ud x=0,
\end{align}
and
\begin{align}\label{final 21}
    \xnm{\re}\ls\oot,
\end{align}
where the $X$ norm is defined in \eqref{ss 00}. Such a solution is unique among all solutions satisfying \eqref{final 31} and \eqref{final 21}. This further yields that in the expansion \eqref{aa 08}, $\m+\e\m\big(\uq_1\cdot\vv\big)$ is the leading-order terms in the sense of
\begin{align}
    \nm{\int_{\r^3}\Big(\fs-\m\Big)}_{L^2_{x}}+\nm{\int_{\r^3}\Big(\fs-\m\Big)\left(\abs{v}^2-3T\right)}_{L^2_{x}}\ls\e,
\end{align}
and 
\begin{align}
    \nm{\int_{\r^3}\Big(\fs-\m-\e \m\big(\uq_1\cdot v\big)\Big)v}_{L^2_{x}}\ls\e^{\frac{3}{2}},
\end{align}
where $(\rq,\uq_1,\tq)$
is determined by the ghost-effect equations \eqref{fluid system-} and \eqref{boundary condition}.
\end{theorem}

\section*{Acknowledgement}

The authors would like to thank Kazuo Aoki and Shigeru Takata for helpful discussions. Also, the authors would like to thank Yong Wang for his comments.

\bibliographystyle{siam}
\bibliography{Reference}

\end{document}